\documentclass[a4paper,10pt]{article}
\usepackage{a4wide}
\usepackage{amsbsy,amscd,amsfonts,amssymb,amstext,amsmath,latexsym,theorem,graphicx}
\usepackage[all]{xy}
\mathchardef\tnode="020E 

\def\arc{
  \hbox{\kern -0.15em
  \vbox{\hrule width 2.5em height 0.6ex depth -0.5 ex}
  \kern -0.33em}}

\def\darc{
  \rlap{\lower0.2ex\arc}{\raise0.2ex\arc}}

\def\stroke#1{
  \kern 0.05em
  \rlap\arc{{\textstyle{#1}}\atop\phantom\arc}
  \kern -0.22em}

\def\dstroke#1{
  \kern 0.05em
  \rlap\darc{{\textstyle{#1}}\atop\phantom\darc}
  \kern -0.22em}

\def\centerscript#1{
  \setbox0=\hbox{$\tnode$}
  \hbox to \wd0{\hss$\scriptstyle{#1}$\hss}}

\def\node{
  \def\super{}
  \def\sub{}
  \futurelet\next\dolabellednode}

  \let\sp=^
  \let\sb=_

  \def\dolabellednode{%
    \ifx\next\sb\let\next\getsub
    \else
      \ifx\next\sp\let\next\getsuper
      \else\let\next\donode
      \fi
    \fi
    \next}

  \def\getsub_#1{\def\sub{#1}\futurelet\next\dolabellednode}
  \def\getsuper^#1{\def\super{#1}\futurelet\next\dolabellednode}

  \def\donode{%
   \rlap{$\mathop{\phantom\tnode}\limits_{\centerscript{\sub}}^{\centerscript{\super}}$}\tnode}

\def\varcdn{
  \kern -0.03em\vbox{\kern -0.5ex
  \hbox to \wd0{\hss\vrule width 0.04em depth 5.8ex\hss}
  \kern -0.3ex  \hbox{$\tnode$}}}

\newcommand{\Theorem}{Theorem}
\newcommand{\Proposition}{Proposition}
\newcommand{\Lemma}{Lemma}
\newcommand{\Corollary}{Corollary}
\newcommand{\Definition}{Definition}
\newcommand{\Remark}{Remark}
\newcommand{\Example}{Example}
\newcommand{\Fact}{Fact}
{\theoremstyle{break}
\newtheorem{theorem}{\Theorem}[section]
\newtheorem{proposition}[theorem]{\Proposition}

\newtheorem{corollary}[theorem]{\Corollary}
\newtheorem{conjecture}[theorem]{Conjecture}
}
{\theorembodyfont{\rmfamily}
\newtheorem{definition}[theorem]{\Definition}
\newtheorem{remark}[theorem]{\Remark}
\newtheorem{example}[theorem]{\Example}
\newtheorem{observation}[theorem]{Observation}
\newtheorem{notation}[theorem]{Notation}
}
\newenvironment{proof}{\noindent {\sl Proof. }}{\hfill $\Box$ \smallskip}
\newenvironment{Proof}[1]{\noindent {\sl Proof #1. }}{\hfill $\Box$ \smallskip}
\newenvironment{introduction}{\begin{quote}
\small\em}{\end{quote}}
\newcommand{\newparagraph}{\noindent \refstepcounter{theorem}{\bf \thetheorem} }
\begin{document}

\title{{\bf On the geometry of global function fields, \\ the Riemann--Roch
theorem, \\ and \\ finiteness properties of $S$-arithmetic groups}}
\author{Ralf Gramlich}

\maketitle

\section{Introduction}

Harder's reduction theory (\cite{Harder:1968}, \cite{Harder:1969}) provides filtrations of Euclidean buildings that allow one to deduce cohomological (\cite{Harder:1977}) and homological (\cite{Stuhler:1980}, \cite{Bux/Wortman}) properties of $S$-arithmetic groups over global function fields.
In this survey I will sketch the main points of Harder's reduction theory, starting from Weil's geometry of numbers and the Riemann--Roch theorem.
I will describe a filtration, used for example in \cite{Behr:1998}, that is 
particularly useful for deriving finiteness properties of $S$-arithmetic
groups. Finally, I will state the recently established rank theorem and some its earlier partial verifications that do 
not restrict the cardinality of the underlying field of constants. As a 
motivation for further research I also state a much more general conjecture 
on isoperimetric properties of $S$-arithmetic groups over global fields 
(number fields or function fields).

\medskip
\noindent
{\bf Acknowledgements.} The author expresses his gratitude to Kai-Uwe Bux, 
Bernhard M\"uhlherr and Stefan Witzel for numerous invaluable discussions on 
the topics of this survey during joint research activities at the Hausdorff 
Institute of Mathematics at Bonn, the MFO at Oberwolfach, and the University 
of Bielefeld. The author also thanks Kai-Uwe Bux, Max Horn, Timoth\'ee
Marquis, Andreas Mars,
Susanne Schimpf, 
Rebecca Waldecker, Markus-Ludwig Wermer, Stefan Witzel and the participants of the reduction theory seminar at Bielefeld during summer 2010, especially Werner Hoffmann and Andrei Rapinchuk, for numerous comments, remarks and suggestions on how to
improve the contents and exposition of this survey. Furthermore, the author thanks two anonymous referees for valuable comments and observations.

\section{Projective varieties}

\begin{introduction}
In this section I give a quick introduction to the concept of projective
varieties. For further reading the sources \cite[I]{Hartshorne:1977} and
\cite[2]{Niederreiter/Xing:2009} are highly recommended.
\end{introduction}

\begin{definition} \label{begin}
Let $k$ be a perfect field, let $\overline{k}$ be its algebraic closure, and let $S \subset \overline{k}[X] = 
\overline{k}[x_0,x_1,...,x_n]$ be a set of {homogeneous} polynomials. The 
set 
$$Z(S) = \{ P \in \mathbb{P}_n(\overline{k}) \mid f(P) = 0 
\mbox{ for all $f \in S$} \}$$ is called a {\bf projective algebraic set}. The algebraic set $Z(S)$ is {\bf defined over $k$}, if $S$ can be chosen to be contained in $k[X]$.
The {\bf Zariski topology} on $\mathbb{P}_n(\overline{k})$ is defined by taking the closed sets to be the projective algebraic sets.

A non-empty projective algebraic set in $\mathbb{P}_n(\overline{k})$ is 
called a {\bf projective variety}, if it is irreducible in the Zariski 
topology of $\mathbb{P}_n(\overline{k})$, i.e., if it is not equal to the 
union of two proper closed subsets. It is called a {\bf projective
$k$-variety} if it is defined over $k$.
The {\bf dimension} of a non-empty projective variety is defined to be its dimension as a topological space in the induced Zariski
topology, i.e., the supremum of all integers $n$ such that there exists a
chain $Z_0 \subsetneq Z_1 \subsetneq \cdots \subsetneq Z_n$ of distinct non-empty
irreducible closed subsets.
\end{definition}

\begin{theorem}[{\cite[2.3.8]{Niederreiter/Xing:2009}}]
A projective algebraic set $V$ is a projective variety if and only if the ideal $I(V)$ of $\overline{k}[X]$ generated by the set $\{ f \in \overline{k}[X] \mid \mbox{ $f$ homogeneous and $f(P)=0$ for all $P \in V$} \}$ is a prime ideal of $\overline{k}[X]$.
\end{theorem}

\begin{definition}
A non-empty intersection of a projective variety in $\mathbb{P}_n(\overline{k})$ with an open subset of $\mathbb{P}_n(\overline{k})$ is called a {\bf quasi-projective variety}.
\end{definition}

Let $V \subseteq \mathbb{P}_n(\overline{k})$ be a quasi-projective variety. For each $P \in V$ there exists a hyperplane $H \subset \mathbb{P}_n(\overline{k})$ with $P \not\in H$. Then $P \in V \backslash H = V \cap A_n(\overline{k})$ for 
$A_n(\overline{k}) := \mathbb{P}_n(\overline{k}) \backslash H$. Let $r$ be a defining relation of the hyperplane $H$ in the variables $x_0, ..., x_n$ (cf.\ \ref{begin}).
A $\overline{k}$-valued function $f$ on $V$ is called {\bf regular at $P$}, if there exists a neighbourhood $N$ of $P$ in $V \cap A_n(\overline{k})$ such that there exist polynomials $a, b \in \overline{k}[X]/(r)$ with $b(Q) \neq 0$ for all $Q \in N$ and $f_{|N} = \frac{a_{|N}}{b_{|N}}$. The function $f$ is {\bf regular} on a non-empty open subset $U$ of $V$, if it is regular at every point of $U$.

\begin{definition}\label{localring}
Let $V$ be a quasi-projective variety and let $P \in V$. Then the 
{\bf local ring $\mathcal{O}_P = \mathcal{O}_P(V)$ at $P$} is defined as 
the ring of germs $[f]$ of functions $f : V \to \overline{k}$ which are regular on a neighbourhood of $P$. In other words, an element of $\mathcal{O}_P$ is an equivalence class of pairs $(U,f)$ where $U$ is an open subset of $Y$ containing $P$, and $f$ is a regular function on $U$, and two such pairs $(U,f)$ and $(V,g)$ are equivalent if $f_{|U \cap V} = g_{|U \cap V}$. For a non-empty open subset $U$ of $V$ define $\mathcal{O}_U = \mathcal{O}_U(V) := \bigcap_{P \in U} \mathcal{O}_P(V)$.
\end{definition}

The ring $\mathcal{O}_P$ is indeed a local ring in the sense of commutative 
algebra: its unique maximal ideal $\mathfrak{m}_P$ is the set of germs of 
regular functions which vanish at $P$. The $i$th power $\mathfrak{m}^i_P$ of 
$\mathfrak{m}_P$ consists of the germs of regular functions whose vanishing 
order at $P$ is at least $i$. Taking these powers $\mathfrak{m}^i_P$ as a 
neighbourhood basis of $\mathcal{O}_P$, 
this defines a topology on $\mathcal{O}_P$, the {\bf $\mathfrak{m}_P$-adic 
topology}. The inverse limit $\hat{\mathcal{O}}_P := \lim_\leftarrow 
\mathcal{O}_P/\mathfrak{m}^i_P$ is the {\bf completion} of $\mathcal{O}_P$; 
cf.\ \cite[7.1]{Eisenbud:1995}, \cite[p.~33]{Hartshorne:1977}, 
\cite[VIII \S 2]{Zariski/Samuel:1975b}. It is a local ring whose maximal 
ideal is denoted by $\hat{\mathfrak{m}}_P$. 

\begin{example} \label{projectiveline}
The projective line $\mathbb{P}_1(\mathbb{C}) \cong \mathbb{C} \cup 
\{ \infty \}$ is a non-singular projective curve (see \ref{curvevar} below).
For $P \in \mathbb{C}$ one has
\begin{eqnarray*}
\mathcal{O}_P & = & \left\{ \frac{a}{b} \in \mathbb{C}(t) \mid b(P) \neq 0 \right\}, \\
\mathfrak{m}_P & = & \left\{ \frac{a}{b} \in \mathcal{O}_P \mid a(P) = 0 \right\}, \\
\hat{\mathcal{O}}_P & = & \left\{ \frac{a}{b} \in \mathbb{C}((t)) \mid b(P) \neq 0 \right\}, \\
\hat{\mathfrak{m}}_P & = & \left\{ \frac{a}{b} \in \hat{\mathcal{O}}_P \mid a(P) = 0 \right\}.
\end{eqnarray*}
\end{example}

\begin{definition}\label{rationalfunctions}
Let $V$ be a variety defined over $k$. Then the 
{\bf field $\overline{k}(V)$ of rational functions of $V$} consists of the 
equivalence classes of pairs $(U,f)$ where $U$ is a non-empty open subset 
of $V$ and $f$ is a regular function on $U$, and two such pairs $(U,f)$ and $(V,g)$ are equivalent if $f_{|U \cap V} = g_{|U \cap V}$. The {\bf field of $k$-rational functions of $V/k$} is $$K := k(V) := \{ h \in \overline{k}(V) \mid \sigma(h) = h \mbox{ for all $\sigma \in \mathrm{Gal}(\overline{k}/k)$} \}.$$
\end{definition}

\begin{definition} \label{closed}
For a point $P \in \mathbb{P}_n(\overline{k})$ the set $\{ \sigma(P) \mid \sigma \in \mathrm{Gal}(\overline{k}/k) \}$ is called a {\bf $k$-closed point}. A $k$-closed point of cardinality $1$ is called {\bf $k$-rational}.
\end{definition}

For each $P \in \mathbb{P}_n(\overline{k})$, each homogeneous polynomial $f \in \overline{k}[X]$, and each  $\sigma \in \mathrm{Gal}(\overline{k}/k)$ one has $f(P) = 0$ if and only if $\sigma(f(P)) = 0$ if and only if $\sigma(f)(\sigma(P)) = 0$. Hence, given a projective $k$-variety $V \subseteq \mathbb{P}_n(\overline{k})$, the assertion $P \in V$ is equivalent to the assertion $\{ \sigma(P) \mid \sigma \in \mathrm{Gal}(\overline{k}/k) \} \subset V$. It therefore makes sense to speak of {\bf $k$-closed points} of $V$. The set of $k$-closed points of $V$ is denoted by $V^\circ = V^\circ/k$. The {\bf degree} $\mathrm{deg}(P)$ of a $k$-closed point $P$ equals the number of points it contains.

\begin{example}
For an involutive automorphism $\sigma$ of $\mathbb{C}$, define $\mathbb{K} := \mathrm{Fix}_\mathbb{C}(\sigma)$. The set of $\mathbb{K}$-closed points of $\mathbb{P}_1(\mathbb{C})$ equals the set of $\sigma$-orbits of $\mathbb{P}_1(\mathbb{C})$, i.e., the $\sigma$-fixed/$\mathbb{K}$-rational points of $\mathbb{P}_1(\mathbb{C})$ and the pairs of $\sigma$-conjugate points of $\mathbb{P}_1(\mathbb{C})$.
\end{example}

\section{Curves over finite fields, considered as varieties}

\begin{introduction}
In this section I turn to one of the main objects of study of this survey,
non-singular projective curves. For further reading the sources 
\cite[I]{Hartshorne:1977}, \cite[3]{Niederreiter/Xing:2009}, and
\cite[II]{Serre:1988} or, if one
prefers a very algebraic approach, \cite[I, II]{Chevalley:1951} and \cite[5,
6]{Rosen:2002} are 
highly recommended.
\end{introduction}

\begin{definition} \label{curvevar} \label{nonsingular}
Let $k$ be a perfect field. A projective variety of dimension $1$ defined over $k$ is called a {\bf projective curve} over $k$
(cf.\ \ref{begin}).
A projective curve $Y$ is called {\bf non-singular} at $P \in Y$, if 
$\mathcal{O}_P$ (cf.\ \ref{localring}) is a discrete valuation ring. It is called {\bf non-singular}, if it is non-singular at each of its points. 
\end{definition}
For $Y$ non-singular at $P$, the valuation of $\mathcal{O}_P$ is given by 
$$\overline{\nu}_P : \mathcal{O}_P \to \mathbb{Z} \cup 
\{ \infty \}: x \mapsto \sup\left\{ i \in \mathbb{N} \cup \{ 0 \} \mid 
x \in \mathfrak{m}^i_P \right\},$$ with the understanding that 
$\mathfrak{m}^0_P = \mathcal{O}_P$. This 
valuation extends to the valuation $\hat{\mathcal{O}}_P \to \mathbb{Z} \cup 
\{ \infty \} : x \mapsto \sup\left\{ i \in \mathbb{N} \cup \{ 0 \} \mid x 
\in \hat{\mathfrak{m}}^i_P \right\}$, as $\hat{\mathfrak{m}}_P 
\cap \mathcal{O}_P = \mathfrak{m}_P$; cf.\ \cite[7.1]{Eisenbud:1995}, 
\cite[I.5.4A]{Hartshorne:1977}. The valuation $\overline{\nu}_P : 
\mathcal{O}_P \to \mathbb{Z} \cup \{ \infty \}$ also extends to a valuation 
on its field of fractions $\overline{k}(Y)$ via $\overline{\nu}_P(x) := \overline{\nu}_P(a) - \overline{\nu}_P(b)$ for $a, b \in \mathcal{O}_P$ with $\frac{a}{b} = x$.

The geometric intuition is that $\overline{\nu}_P(x)$ indicates whether $P$ is a zero or a pole of $x$ and counts its multiplicity (cf.\ \ref{projectiveline}).

\medskip
An {\bf algebraic function field} (in one variable over $k$) is an extension field $\mathbb{F}$ of $k$ that admits an element $x$ that is transcendental over $k$ such that $\mathbb{F}/k(x)$ is a field extension of finite degree.    

\begin{theorem}[{\cite[3.2.9]{Niederreiter/Xing:2009}}]
There is a one-to-one correspondence between $k$-isomorphism classes of 
non-singular projective curves over $k$ and $k$-isomorphism classes of 
algebraic function fields of one variable with full constant field $k$, via 
the map $Y/k \mapsto k(Y)$ (cf.\ \ref{rationalfunctions}).
\end{theorem}
Here, the {\bf full constant field} of an algebraic function field $K$ over 
$k$ is the algebraic closure of $k$ in $K$. In case the full constant field 
is finite, $K$ is a {\bf global function field}. For example, $\mathbb{F}_q(t)$
is a global function field with full constant field $\mathbb{F}_q$.   

\begin{definition} \label{place}
Two discrete valuations $\nu_1$ and $\nu_2$ of an algebraic function field $K$ are {\bf equivalent} if there exists a constant $c > 0$ such that $\nu_1(x) = c\nu_2(x)$ for all $0 \neq x \in K$. An equivalence class of discrete valuations is called a {\bf place}. The {\bf degree} of a valuation/place is the degree of the residue class field $\mathcal{O}_P/\mathfrak{m}_P$ over the constant field $k$. This is always a finite number; cf.\ \cite[1.5.13]{Niederreiter/Xing:2009}. 
\end{definition}

\begin{theorem}[{\cite[3.1.15]{Niederreiter/Xing:2009}}]
Let $Y/\mathbb{F}_q$ be a non-singular projective curve. Then there exists a 
natural one-to-one correspondence between $\mathbb{F}_q$-closed points 
of $Y$ and places (cf.\ \ref{place}) of the field $K=\mathbb{F}_q(Y)$ of $\mathbb{F}_q$-rational functions (cf.\ \ref{rationalfunctions}). Moreover, the degree of an $\mathbb{F}_q$-closed point is equal to the degree of the corresponding place.
\end{theorem}

\newparagraph \label{valuation}
Given a point $P$ in an $\mathbb{F}_q$-closed point of $Y$, the valuation 
$\overline{\nu}_P : \mathcal{O}_P \to \mathbb{Z} \cup \{ \infty \}$ 
(cf.\ \ref{nonsingular})  restricts to a valuation $\nu_P : 
\mathcal{O}_{P,K} := \mathcal{O}_P \cap \mathbb{F}_q(Y) \to \mathbb{Z} 
\cup \{ \infty \}$. By \cite[3.1.14]{Niederreiter/Xing:2009} this valuation 
only depends on the $\mathbb{F}_q$-closed point of $Y$ containing $P$ and 
not on the particular choice of $P$ inside that $\mathbb{F}_q$-closed point. 
As before, $\nu_P$ extends to the field of $\mathbb{F}_q$-rational functions 
$K = \mathbb{F}_q(Y)$, the completion $\hat{\mathcal{O}}_{P,K} := 
\widehat{\mathcal{O}_P \cap K}$, and to the field of fractions $K_P$ of the completion
$\hat{\mathcal{O}}_{P,K}$. The field $K_P$ is the {\bf local function field at $P$}.
For example, $\mathbb{F}_q((t))$ is the local function field of the global 
function field $\mathbb{F}(t)$ at the place corresponding to the irreducible
polynomial $t$.

\begin{definition} \label{divisor}
Let $Y/\mathbb{F}_q$ be a non-singular projective curve. The {\bf Weil divisor group}
$\mathrm{Div}(Y)=\mathrm{Div}(Y/\mathbb{F}_q)$ is the free abelian group over the set $Y^\circ$ of $\mathbb{F}_q$-closed points of $Y$ (cf.\ \ref{closed}).
An element $D = \sum_{P \in Y^\circ} n_P P \in \mathrm{Div}(Y)$ is called a {\bf Weil divisor} of $Y$. It is {\bf effective}, in symbols $D \geq 0$, if $n_P \geq 0$ for all $P \in Y^\circ$. For two divisors $D_1$ and $D_2$ of $Y$ one writes $D_1 \geq D_2$, if $D_1 - D_2 \geq 0$.
The {\bf degree} $\mathrm{deg}(D)$ of a Weil divisor 
$D = \sum_{P \in Y^\circ} n_P P$ is given by 
$\mathrm{deg}(D) := \sum_{P \in Y^\circ} n_P \mathrm{deg}(P)$ (cf.~\ref{closed}). Also, define $\nu_P(D) := n_P$.
\end{definition}
For $0 \neq x \in K=\mathbb{F}_q(Y)$, define the {\bf divisor} 
$\mathrm{div}(x)$ of $x$ by $\mathrm{div}(x) := \sum_{P \in Y^\circ} 
\nu_P(x) P$. As, by \cite[3.3.2]{Niederreiter/Xing:2009},
\cite[5.1]{Rosen:2002}, any $0 \neq x \in K$ admits only finitely many zeros 
(i.e., points $P \in Y^\circ$ with $\nu_P(x) > 0$) and poles (i.e., 
points $P \in Y^\circ$ with $\nu_P(x) < 0$), the divisor $\mathrm{div}(x)$ 
indeed is a Weil divisor. A Weil divisor 
obtained in this way is {\bf principal}, and has degree $0$; cf.~\cite[3.4.3]{Niederreiter/Xing:2009}, \cite[5.1]{Rosen:2002}. Two Weil divisors that differ by a principal divisor are called {\bf equivalent}.

\begin{definition}
For a divisor $D$ of $Y/\mathbb{F}_q$, the {\bf Riemann--Roch space} $L(D)$ is defined
as $$L(D) := \{ 0 \neq x \in K = \mathbb{F}_q(Y) \mid \mathrm{div}(x) + D \geq 0 \} \cup \{ 0 \}.$$ 
\end{definition}
The Riemann--Roch space $L(D)$ of a divisor $D$ is an $\mathbb{F}_q$-vector space of finite dimension (cf.\ \cite[3.4.1(iv)]{Niederreiter/Xing:2009}).

\begin{definition} \label{incompleteadeles}
Define the {\bf ring of repartitions} as $$\mathbb{A}_K := \{ (x_P)_{P \in Y^{\circ}} \in \prod_{P \in Y^{\circ}} K \mid x_P \in \mathcal{O}_P \mbox{ for almost all $P\in Y^\circ$ } \} .$$
By \cite[p.~25]{Chevalley:1951}, \cite[3.3.2]{Niederreiter/Xing:2009}, the field $K$ embeds diagonally in $\mathbb{A}_K$.
For any divisor $D \in \mathrm{Div}(Y)$ define $\mathbb{A}_K(D) := 
\{ x \in \mathbb{A}_K \mid \nu_P(x_P) + \nu_P(D) \geq 0 \mbox{ for all 
$P \in Y^\circ$} \}$. Note that $\mathbb{A}_K(D)$ is an
$\mathbb{F}_q$-subvector space of $\mathbb{A}_K$.
\end{definition}

\begin{definition}
A {\bf Weil differential} of a non-singular projective curve 
$Y/\mathbb{F}_q$, resp.\ of its global function field $K$, is an $\mathbb{F}_q$-linear map $\omega : \mathbb{A}_K \to \mathbb{F}_q$ such that $\omega_{|\mathbb{A}_K(D)+K} = 0$ for some divisor $D \in \mathrm{Div}(Y)$. Denote by $\Omega_K$ the set of all Weil differentials of $K$ and by $\Omega_K(D)$ the set of all Weil differentials of $K$ that vanish on $\mathbb{A}_K(D) + K$.

\end{definition}
Note the difference between the ring of repartitions $\mathbb{A}_K$ and the ring of ad\`eles $\hat{\mathbb{A}}_K$ defined in \ref{adeles}
below (completions). The concept of a Weil differential can equally well be 
introduced using the ring of ad\`eles $\hat{\mathbb{A}}_K$, cf.\
\cite{Rosen:2002}.

\begin{observation} \label{omegaadel}
Let $D \in \mathrm{Div}(Y)$. Then $\Omega_K(D) \cong \mathbb{A}_K/\mathbb{A}_K(D) + K$ as $\mathbb{F}_q$-vector spaces. 
\end{observation}
If $0 \neq \omega \in \Omega_K$, then by
\cite[3.6.11]{Niederreiter/Xing:2009}, \cite[6.8]{Rosen:2002} the set $\{ D \in \mathrm{Div}(Y) \mid \omega_{|\mathbb{A}_K(D) + K} = 0 \}$ has a unique maximal element with respect to $\geq$.
This maximal element is called a {\bf canonical divisor} of $Y$ and denoted by $(\omega)$. By
\cite[3.6.10]{Niederreiter/Xing:2009}, \cite[6.10]{Rosen:2002} 
\begin{eqnarray}
\mathrm{dim}_K(\Omega_K) & = & 1 \label{dim1}
\end{eqnarray}
 and by \cite[3.6.13]{Niederreiter/Xing:2009}, \cite[6.9]{Rosen:2002} the Weil differential $x\omega : \mathbb{A}_K \to \mathbb{F}_q : a \mapsto \omega(xa)$ satisfies
\begin{eqnarray}
(x\omega) = \mathrm{div}(x) + (\omega) \label{mult}
\end{eqnarray}
 for any $0 \neq x \in K$, $0 \neq \omega \in \Omega_K$. Hence any two canonical divisors of $Y$ are equivalent (cf.\ \ref{divisor}).

\begin{observation} \label{lomega}
Let $D$ be a divisor of $Y/\mathbb{F}_q$ and let $W = (\omega)$ be a canonical divisor of $Y$. Then $$L(W-D) \to \Omega_K(D) : x \mapsto x\omega$$ is an isomorphism of $\mathbb{F}_q$-vector spaces.
\end{observation}

\begin{proof}
For $0 \neq x \in L(W-D)$ we have\footnote{The notation $\stackrel{{(\ref{mult})}}{=}$ means that (\ref{mult}) is a justification for the equality.} $(x\omega) \stackrel{{(\ref{mult})}}{=} \mathrm{div}(x) + (\omega) \geq -(W-D)+W = D$, so that indeed $x\omega \in \Omega_K(D)$. 
Injectivity and $\mathbb{F}_q$-linearity of the map $x \mapsto x\omega$ are clear.  In order to prove surjectivity, let $0\neq \omega_1 \in \Omega_K(D)$. By (\ref{dim1}) there exists $0 \neq x \in K$ such that $\omega_1 = x\omega$. Since
we have $\mathrm{div}(x) + W \stackrel{(\ref{mult})}{=} (x\omega) = (\omega_1)
 \geq D$, it follows that $x \in L(W-D)$. 
\end{proof}

\begin{definition} \label{genus}
Let $Y/\mathbb{F}_q$ be a non-singular projective curve and let $W$ be a
canonical divisor of $Y$. Define the {\bf genus} $g$ of the curve $Y/\mathbb{F}_q$ by $$g := \mathrm{dim}_{\mathbb{F}_q}(L(W)) \stackrel{\ref{lomega}}{=}
\mathrm{dim}_{\mathbb{F}_q}(\Omega_K(0)) \stackrel{\ref{omegaadel}}{=} \mathrm{dim}_{\mathbb{F}_q}(\mathbb{A}_K/\mathbb{A}_K(0) + K).$$
\end{definition}

\begin{theorem}[Riemann--Roch Theorem,
{\cite[p.~30]{Chevalley:1951}, \cite[3.6.14]{Niederreiter/Xing:2009}, \cite[5.4]{Rosen:2002}}] \label{RiemannRoch}
Let $Y/\mathbb{F}_q$ be a non-singular projective curve of genus $g$ and let 
$W$ be a canonical divisor. Then for any divisor $D \in \mathrm{Div}(Y)$ one
has $$\mathrm{dim}_{\mathbb{F}_q}(L(D)) - \mathrm{dim}_{\mathbb{F}_q}(L(W-D)) = \mathrm{deg}(D) + 1 - g.$$
\end{theorem}
The Riemann--Roch theorem combined with the study of the zeta function of
the global function field $\mathbb{F}_q(Y)$ allows one to establish the
following useful result.

\begin{proposition}[{\cite[4.1.10]{Niederreiter/Xing:2009}}]\label{invertible1}
Let $Y/\mathbb{F}_q$ be a non-singular projective curve. Then there exists 
a Weil divisor of $Y$ of degree~$1$.
\end{proposition}
Note that this Weil divisor of degree $1$ need not be effective. In fact,
for $p \geq 5$ prime the non-singular projective curve over $\mathbb{F}_p$ given by $x^{p-1} +
y^{p-1} = 3z^{p-1}$ does not admit any $\mathbb{F}_p$-rational point and,
thus, there cannot exist an effective Weil divisor of degree $1$ of this
curve.

\begin{definition} \label{vectorbundle}
A {\bf vector bundle} of rank $r$ over a curve $Y/k$ is a variety $E/k$ 
together with a morphism $\pi : E \to Y$ such that there exists an open 
covering $\{ U_i \}$ of $Y$ and isomorphisms 
$\phi_i : \pi^{-1}(U_i) \to U_i \times A_r(k)$ (where $A_r(k)$ denotes the 
affine space over $k$ of dimension $r$) such that for each pair $U_i$, $U_j$ 
the composition ${\phi_j {\phi_i}^{-1}}_{|U_i \cap U_j}$ equals $(\mathrm{id}, \phi_{i,j})$ for a linear map $\phi_{i,j}$.
A {\bf section} of a vector bundle $\pi : E \to Y$ over an open set $U \subset Y$ is a map $s : U \to E$ such that $\pi \circ s = \mathrm{id}$.
\end{definition}

\section{Geometry of numbers}

\begin{introduction}
In this section I describe Weil's geometry of numbers. For further reading 
the sources 
\cite[2]{Weil:1982}, \cite[VI]{Weil:1995} are 
highly recommended.
\end{introduction}

\newparagraph \label{defnorm}
Let $Y/\mathbb{F}_q$ be a non-singular projective curve, let $P$ be an 
$\mathbb{F}_q$-closed point of $Y$, and let $\nu_P$ be the valuation of 
$K = \mathbb{F}_q(Y)$ discussed in \ref{valuation}. Then 
$$|\cdot|_P : K \to \mathbb{R} : x \mapsto (q^{\mathrm{deg}(P)})^{-\nu_P(x)}$$
defines an absolute value on $K$. The completion of $K$ with respect to this 
absolute value equals $K_P$ (cf.~\ref{valuation}), which is locally compact as the zero neighbourhood $\hat{\mathcal{O}}_{P,K}$ is an inverse limit of finite rings (cf.~\ref{localring}), whence compact. On each field $K_P$ define the canonical Haar measure
to be the one with respect to which $\hat{\mathcal{O}}_{P,K}$ has volume $1$.

\begin{theorem}[\cite{Artin/Whaples:1945}] \label{productformula}
Let $Y/\mathbb{F}_q$ be a non-singular projective curve. Then each $0 \neq x \in K$ satisfies $|x|_P = 1$ for almost all $P \in Y^\circ$. Moreover, for each $0 \neq x \in K$, $$\prod_{P \in Y^\circ} |x|_P = 1.$$
\end{theorem}

\begin{definition} \label{adeles}
In analogy to \ref{incompleteadeles} define the {\bf ad\`ele ring} $$\hat{\mathbb{A}}_K := \{ (x_P)_{P \in Y^{\circ}} \in \prod_{P \in Y^{\circ}} K_P \mid x_P \in \hat{\mathcal{O}}_P \mbox{ for almost all $P\in Y^\circ$ } \} .$$
It is a locally compact ring and contains $K$ embedded diagonally as a discrete subring (cf.\ \ref{productformula}). Define $$|\cdot| : \hat{\mathbb{A}}_K \to \mathbb{R} : x \mapsto \prod_{P \in Y^{\circ}} |x|_P.$$ The product measure of the canonical Haar measures on each individual $K_P$ (\ref{defnorm}) yields a Haar measure on the locally compact group 
$(\hat{\mathbb{A}}_K,+)$. Let $\omega_{\hat{\mathbb{A}}_K}$ denote the 
corresponding volume form. For any divisor $D \in \mathrm{Div}(Y)$, define 
$\hat{\mathbb{A}}_K(D) := \{ x \in \hat{\mathbb{A}}_K \mid \nu_P(x_P) + 
\nu_P(D) \geq 0 \mbox{ for all $P \in Y^\circ$} \}$. Note that 
$\hat{\mathcal{O}}_K := \hat{\mathbb{A}}_K(0)$ is the maximal compact subring 
of $\hat{\mathbb{A}}_K$. It has volume $1$ with respect to the chosen Haar
measure.
\end{definition}

\newparagraph \label{sadeles}
For a finite subset $S \subset Y^\circ$ define the {\bf ring of $S$-ad\`eles} 
$$\hat{\mathbb{A}}_S := \prod_{P \in S} K_P \times \prod_{P \in Y^\circ 
\backslash S} \hat{\mathcal{O}}_{P,K}$$ (cf.~\ref{valuation}). One has 
$\hat{\mathbb{A}}_K = \lim_\rightarrow \hat{\mathbb{A}}_S$. For $K$ embedded 
diagonally in $\hat{\mathbb{A}}_K$ as a discrete subring (cf.~\ref{adeles}) define the 
{\bf ring of $S$-integers} $$\mathcal{O}_S := K \cap \hat{\mathbb{A}}_S = 
\bigcap_{P \in Y^\circ \backslash S} (K \cap \hat{\mathcal{O}}_{P,K})
\stackrel{\ref{valuation}}{=}
\bigcap_{P \in Y^\circ \backslash S} \mathcal{O}_{P,K} 
\stackrel{\ref{localring}}{=} \mathcal{O}_{Y^\circ \backslash S, K};$$ these are the $k$-rational functions of $Y$ which are regular outside $S$.
The field $K$ embeds diagonally in $\prod_{P \in S} K_P$ as a dense
subring and, hence, also in $\mathbb{A}_S$. Define 
$$|\cdot|_S : \prod_{P \in S} K_P \to \mathbb{R} : x \mapsto \prod_{P \in S} 
|x|_P.$$

\medskip
\newparagraph \label{serre}
Let $G$ be a unimodular locally compact group, let $\Gamma$ be a discrete 
subgroup of $G$, let $H$ be a compact open subgroup of $G$, let $\mu$ be a 
Haar measure on $G$, and assume the double coset space $X \cong H \backslash G / \Gamma$, considered as a system of 
representatives of the $\Gamma$-orbits on $H\backslash G$, is countable. Note that $\mu$ is also a Haar measure on $H$, since $H$ is open, and hence of positive volume.  Then, cf.~\cite[p.~84]{Serre:2003}, 
$$\int_{G/\Gamma} d\mu = \sum_{x \in X} \left(\int_{G_{x}/\Gamma_{x}} 
d\mu\right) = \sum_{x\in X} \frac{\int_{G_{x}} d\mu}{|\Gamma_{x}|} 
= \sum_{x\in X} \frac{\int_{H} d\mu}{|\Gamma_{x}|} 
= \int_{H} d\mu \sum_{x\in X} \frac{1}{|\Gamma_{x}|}.$$
Here, for choices of $x_0 \in X$ and $(g^x_{x_0})_{x \in X} \in G$ with
$g^x_{x_0}(x_0)=x$, the map $\bigsqcup_{x \in X} G_x/\Gamma_x \to G/\Gamma : g\Gamma_x
\mapsto g^{x}_{x_0}g\Gamma$ is an isomorphism of orbit spaces, where $G_x$
and $\Gamma_x$ denote the stabilizers of $x$ in the respective group.

\medskip
\newparagraph\label{12}
One has 
\begin{eqnarray}
& \hat{\mathbb{A}}_K(D) \backslash \hat{\mathbb{A}}_K/ K = \hat{\mathbb{A}}_K/ \hat{\mathbb{A}}_K(D) + K \cong \mathbb{A}_K / \mathbb{A}_K(D) + K \stackrel{\ref{omegaadel}}{\cong} \Omega_K(D) \stackrel{\ref{lomega}}{\cong} L(W-D), & \label{1}\\ 
& K \cap \hat{\mathbb{A}}_K(D) = K \cap \{ x \in \hat{\mathbb{A}}_K \mid \nu_P(x) + \nu_P(D) \geq 0 \mbox{ for all $P \in Y^\circ$} \} = L(D). & \label{2}
\end{eqnarray}
The observation that the intersection of the lattice $K$ with the compactum $\hat{\mathbb{A}}_K(D)$ equals $L(D)$ is a classical theme in the geometry of numbers, cf.\ \cite[p.~387]{Armitage:1967}.

\begin{proposition}[{\cite[2.1.3]{Weil:1982}, \cite[p.~100]{Weil:1995}}] \label{mu}
Using the notation of \ref{adeles}, one has $$\int_{\hat{\mathbb{A}}_K/K} \omega_{\hat{\mathbb{A}}_K} = q^{g-1}.$$
\end{proposition}

\begin{proof}
For the compact open subgroup $\hat{\mathbb{A}}_K(0)$ of $\hat{\mathbb{A}}_K$ Serre's formula (cf.\ \ref{serre}) yields
\begin{equation}
\int_{\hat{\mathbb{A}}_K/K} \omega_{\hat{\mathbb{A}}_K} = \int_{\hat{\mathbb{A}}_K(0)} \omega_{\hat{\mathbb{A}}_K} \sum_{\hat{\mathbb{A}}_K(0) \backslash \hat{\mathbb{A}}_K/ K} \frac{1}{|K \cap \hat{\mathbb{A}}_K(0)|} = \frac{|\hat{\mathbb{A}}_K(0) \backslash \hat{\mathbb{A}}_K/ K|}{|K \cap \hat{\mathbb{A}}_K(0)|} \int_{\hat{\mathbb{A}}_K(0)} \omega_{\hat{\mathbb{A}}_K}. \label{applicationofserre}
\end{equation}
Therefore, as $\int_{\hat{\mathbb{A}}_K(0)} \omega_{\hat{\mathbb{A}}_K} = 1$
(cf.~\ref{adeles}), the Riemann--Roch Theorem (cf.~\ref{RiemannRoch}) or, in fact, its underlying definitions and observations \ref{genus} and \ref{12} imply $\int_{\hat{\mathbb{A}}_K/K} \omega_{\hat{\mathbb{A}}_K} = q^{g-1}$.
\end{proof}

\begin{proposition}[{\cite[p.~39]{Harder:1969}, \cite[p.~98]{Weil:1995}}] \label{mu2}
One has $$\int_{\hat{\mathbb{A}}_K(D)} \omega_{\hat{\mathbb{A}}_K} = q^{\mathrm{deg}(D)}.$$
\end{proposition}

\begin{proof}
The strategy is to first use \ref{mu} and then apply \ref{serre} to the compact open subgroup $\hat{\mathbb{A}}_K(D)$ of $\hat{\mathbb{A}}_K$ in analogy to (\ref{applicationofserre}) in the proof of \ref{mu}. This way one computes
\begin{eqnarray*}
q^{g-1} & \stackrel{\ref{mu}}{=} & \int_{\hat{\mathbb{A}}_K/K} \omega_{\hat{\mathbb{A}}_K} \\
& \stackrel{\ref{serre}}{=} & \frac{|\hat{\mathbb{A}}_K(D) \backslash \hat{\mathbb{A}}_K/ K|}{|K \cap \hat{\mathbb{A}}_K(D)|} \int_{\hat{\mathbb{A}}_K(D)} \omega_{\hat{\mathbb{A}}_K} \\
& \stackrel{\ref{12}}{=} & \frac{|L(W-D)|}{|L(D)|} \int_{\hat{\mathbb{A}}_K(D)} \omega_{\hat{\mathbb{A}}_K}.
\end{eqnarray*}
Hence, by the Riemann--Roch Theorem (cf.~\ref{RiemannRoch}), one has $\int_{\hat{\mathbb{A}}_K(D)} \omega_{\hat{\mathbb{A}}_K}= q^{\mathrm{deg}(D)}.$
\end{proof}

\section{Curves over finite fields, considered as schemes}

\begin{introduction}
In this section I return to the study of non-singular projective curves, in
the more general context of schemes. For further reading 
the sources 
\cite[II, III, IV]{Hartshorne:1977}, \cite[2, 6, 7]{Liu:2002} are 
highly recommended.
\end{introduction}

The basic idea of Harder's reduction theory is to apply the geometry of 
numbers to the groups of $\hat{\mathbb{A}}_K$-rational points of reductive $K$-isotropic algebraic 
$K$-groups, see \ref{GeometryOfNumbers} below. The concept of schemes, which 
generalizes the concept of varieties, allows one to do so in a very 
efficient way, as in this language one can consider reductive groups over
projective curves, thus making the Riemann--Roch theorem applicable. 

\begin{definition} \label{stalk}
Let $X$ be a topological space, let $\mathfrak{Top}(X)$ be the category of 
open subsets of $X$ with the inclusion maps as morphisms, and let 
$\mathfrak{C}$ be one of the categories $\mathfrak{Ab}$ of abelian groups, 
$\mathfrak{CRing}$ of commutative rings with $1$, or $\mathfrak{Mod}(A)$ of modules over the (commutative) ring $A$. A contravariant functor $\mathcal{F}$ from $\mathfrak{Top}(X)$ to $\mathfrak{C}$ which satisfies $\mathcal{F}(\emptyset) = 0$ is called a {\bf presheaf} on $X$ with values in $\mathfrak{C}$. For presheaves $\mathcal{F}$, $\mathcal{G}$ on $X$ with values in $\mathfrak{C}$, a natural transformation from $\mathcal{F}$ to $\mathcal{G}$ is called a {\bf morphism} of presheaves.

A presheaf $\mathcal{F}$ is called a {\bf sheaf}, if it satisfies the following conditions: (a) if $U$ is an open subset of $X$, if $\{ V_i \}$ is an open covering of $U$, and if $s \in \mathcal{F}(U)$ satisfies $s_{|V_i} = 0$ for all $i$, then $s=0$; and (b) if $U$ is an open subset of $X$, if $\{ V_i \}$ is an open covering of $U$, and if the $s_i \in \mathcal{F}(V_i)$ satisfy ${s_i}_{|V_i \cap V_j} = {s_j}_{|V_i \cap V_j}$ for all $i$, $j$, then there exists $s \in \mathcal{F}(U)$ such that $s_{|V_i} = s_i$ for all $i$.

For a presheaf $\mathcal{F}$ on $X$ and an element $x \in X$, the 
{\bf stalk} $\mathcal{F}_x$ is defined as the direct limit 
$\lim_\rightarrow \mathcal{F}(U)$ over the open neighbourhoods $U$ of $x$ 
where $U \leq V$ if $U \supseteq V$, i.e., $\mathcal{F}_x$ consists of the 
germs of elements of $\mathcal{F}(U)$ at $x$.
\end{definition}

\begin{example} \label{curveasscheme}
Let $V$ be a quasi-projective variety. The functor $\mathcal{O}_V : U \mapsto \mathcal{O}_U(V)$ (cf.\ \ref{localring}) is a sheaf on $V$. The stalk $\mathcal{O}_{V,P}$ at a point $P$ equals the local ring $\mathcal{O}_P$ at that point.  
\end{example}

\begin{definition} \label{sheafification}
Let $X$ be a topological space and let $\mathcal{F}$ be a presheaf on $X$. The {\bf sheaf associated to the presheaf $\mathcal{F}$}, also called the {\bf sheafification of $\mathcal{F}$}, is a pair $(\mathcal{F}^\dagger,\theta)$ consisting of a sheaf $\mathcal{F}^\dagger$ and a morphism $\theta : \mathcal{F} \to \mathcal{F}^\dagger$ of presheaves that satisfies the following universal property: for every morphism $\alpha : \mathcal{F} \to \mathcal{G}$, where $\mathcal{G}$ is a sheaf, there exists a unique morphism $\tilde{\alpha} : \mathcal{F}^\dagger \to \mathcal{G}$ such that $\alpha = \tilde{\alpha} \circ \theta$.   
\end{definition}
The sheafification of a presheaf is unique for abstract reasons and exists by \cite[2.2.15]{Liu:2002}; see also \cite[3.3]{Harder:2008}, \cite[II.1.2]{Hartshorne:1977}, \cite[Ex.\ 2.2.3]{Liu:2002}.

\begin{definition} \label{quotient}
A subfunctor $\mathcal{F}'$ of a sheaf $\mathcal{F}$ that is itself a sheaf is called a {\bf subsheaf} of $\mathcal{F}$. If $\mathcal{F}$ takes values in $\mathfrak{Ab}$, then the {\bf quotient sheaf} $\mathcal{F}/\mathcal{F}'$ is the sheaf associated to the presheaf $U \mapsto \mathcal{F}(U)/\mathcal{F}'(U)$
(cf.~\ref{sheafification}). If $\mathcal{F}$ is a sheaf on $X$ then, given an open cover $\{ U_i \}$ of $X$, an element of $\mathcal{F}/\mathcal{F}'(X)$ can be described by elements $f_i \in \mathcal{F}(U_i)$ such that $\frac{f_i}{f_j} \in \mathcal{F}'(U_i \cap U_j)$ for all $i$, $j$.
\end{definition}

\begin{definition}\label{directimage}
Let $f : X \to Y$ be a continous map between topological spaces, let $\mathcal{F}$ be a sheaf on $X$, and let $\mathcal{G}$ be a sheaf on $Y$. Then $Y \supseteq U \mapsto \mathcal{F}(f^{-1}(U))$ defines a sheaf $f_*\mathcal{F}$ on $Y$, the {\bf direct image sheaf}. For each $x \in X$ there is a natural map $\epsilon_x : (f_*\mathcal{F})_{f(x)} \to \mathcal{F}_{x}$. The {\bf inverse image sheaf} $f^{-1}\mathcal{G}$ on $X$ is the sheaf associated to the presheaf $X \supseteq U \mapsto \lim_{V \supseteq f(U)} \mathcal{G}(V)$ where the limit is taken over all open subsets $V$ of $Y$ that contain $f(U)$; cf.\ \ref{sheafification}. For each $x \in X$ one has $(f^{-1}\mathcal{G})_x = \mathcal{G}_{f(x)}$ (\cite[p.~37]{Liu:2002}).
\end{definition}

\begin{definition}
A {\bf ringed space} is a pair $(X,\mathcal{O}_X)$ consisting of a 
topological space $X$ and a sheaf of rings $\mathcal{O}_X$, i.e., a sheaf on
$X$ with values in $\mathfrak{CRing}$. It is a 
{\bf locally ringed space}, if for each point $P \in X$, the stalk 
$\mathcal{O}_{X,P}$ is a local ring.
\end{definition}

\begin{example}
Let $A$ be a commutative ring, define the set $\mathrm{Spec}(A)$ to be the set of prime ideals of $A$, and define a topology on $\mathrm{Spec}(A)$ by taking the closed sets to be sets of the type $V(\mathfrak{a}) = \{ \mathfrak{p} \in \mathrm{Spec}(A) \mid \mathfrak{a} \subseteq \mathfrak{p} \}$, for arbitrary ideals $\mathfrak{a}$ of $A$; cf.\ \cite[II.2.1]{Hartshorne:1977}, \cite[2.1.1]{Liu:2002}. 
For an open set $U \subseteq \mathrm{Spec}(A)$ define $\mathcal{O}(U)$ to be the set of functions $s : U \to \bigsqcup_{\mathfrak{p} \in U} A_\mathfrak{p}$, where $A_\mathfrak{p}$ denotes the localization of $A$ at $\mathfrak{p}$, such that $s(\mathfrak{p}) \in A_\mathfrak{p}$ and for each $\mathfrak{p} \in U$ there exist a neighbourhood $V \subseteq U$ of $\mathfrak{p}$ and elements $a, f \in A$ with $f \not\in \mathfrak{q}$ and $s(\mathfrak{q}) = \frac{a}{f}$ in $A_\mathfrak{q}$ for each $\mathfrak{q} \in V$. The resulting locally ringed space $\mathrm{Spec} A = (\mathrm{Spec}(A),\mathcal{O})$ is called the {\bf spectrum} of $A$; cf.\ \cite[p.~70]{Hartshorne:1977}, \cite[2.3.2]{Liu:2002}.
\end{example}

\begin{definition} \label{morphring}
A {\bf morphism of locally ringed spaces} $(f,f^\sharp) : (X,\mathcal{O}_X) \to (Y,\mathcal{O}_Y)$ consists of a continuous map $f : X \to Y$ and a morphism of sheaves of rings $f^\sharp : \mathcal{O}_Y \to f_* \mathcal{O}_X$ such that for every $x \in X$ the induced map $\mathcal{O}_{Y,f(x)} \stackrel{f^\sharp}{\to} (f_*\mathcal{O}_X)_{f(x)} \stackrel{\epsilon_x}{\to} \mathcal{O}_{X,x}$ is a local homomorphism (cf.\ \ref{directimage}). An {\bf isomorphism} is an invertible morphism. A morphism $(f,f^\sharp) : (Z,\mathcal{O}_Z) \to (X,\mathcal{O}_X)$ of ringed spaces is a {\bf closed immersion} if $f$ is a topological closed immersion and if each $f^\sharp_x$ is surjective. 
\end{definition}

\begin{definition} \label{oxmodule}
Let $(X,\mathcal{O}_X)$ be a ringed space. An {\bf $\mathcal{O}_X$-module} is a sheaf $\mathcal{F}$ on $X$ with values in $\mathfrak{Ab}$ such that for each open subset $U \subseteq X$ the group $\mathcal{F}(U)$ is an $\mathcal{O}_X(U)$-module and, for each inclusion of open sets $V \subseteq U$, the group homomorphism $\mathcal{F}(U) \to \mathcal{F}(V)$ is compatible with the module structures via the ring homomorphism $\mathcal{O}_X(U) \to \mathcal{O}_X(V)$.

If $U$ is an open subset of $X$, and if $\mathcal{F}$ is an $\mathcal{O}_X$-module, then $\mathcal{F}_{|U}$ is an $\mathcal{O}_{X|U}$-module. If $\mathcal{F}$ and $\mathcal{G}$ are two $\mathcal{O}_X$-modules, the group of morphisms from $\mathcal{F}$ to $\mathcal{G}$ is denoted by $\mathrm{Hom}_{\mathcal{O}_X}(\mathcal{F},\mathcal{G})$. The presheaf $U \mapsto \mathrm{Hom}_{\mathcal{O}_{X|U}}(\mathcal{F}_{|U}, \mathcal{G}_{|U})$ is a sheaf, denoted by $\mathcal{H}om_{\mathcal{O}_X}(\mathcal{F},\mathcal{G})$; it is also an $\mathcal{O}_X$-module.

The {\bf tensor product} $\mathcal{F} \otimes_{\mathcal{O}_X} \mathcal{G}$ of 
two $\mathcal{O}_X$-modules is defined as the sheaf associated to the presheaf $U \mapsto \mathcal{F}(U) \otimes_{\mathcal{O}_X(U)} \mathcal{G}(U)$ (cf.\
\ref{sheafification}).

An $\mathcal{O}_X$-module $\mathcal{F}$ is {\bf free} if it is isomorphic 
to a direct sum of copies of $\mathcal{O}_X$. It is {\bf locally free} if 
there exists an open cover $\{ U_i \}$ of $X$ for which each $\mathcal{F}_{|U_i}$ is a free $\mathcal{O}_X$-module. In that case the {\bf rank} of $\mathcal{F}$ on such an open set is the number of copies of the structure sheaf $\mathcal{O}_X$ needed. For $X$ irreducible, thus, the notion of rank of a locally free sheaf is well defined. A locally free $\mathcal{O}_X$-module of rank $1$ is called an {\bf invertible $\mathcal{O}_X$-module}.
\end{definition}

\begin{definition}
For any ringed space $(X,\mathcal{O}_X)$, define the {\bf Picard group} $\mathrm{Pic(X)}$ of $X$ to be the group of isomorphism classes of invertible $\mathcal{O}_X$-modules under the operation $\otimes_{\mathcal{O}_X}$, cf.\ \cite[II.6.12]{Hartshorne:1977}. The inverse of an invertible $\mathcal{O}_X$-module $\mathcal{L}$ is $\mathcal{L}^\vee := \mathcal{H}om_{\mathcal{O}_X}(\mathcal{L},\mathcal{O}_X)$, cf.\ \cite[Ex.~II.5.1]{Hartshorne:1977}, \cite[Ex.~5.1.12]{Liu:2002}.
\end{definition}

\begin{definition}
An {\bf affine scheme} is a locally ringed space $(X,\mathcal{O}_X)$ which is isomorphic to the spectrum of some ring. A locally ringed space $(X,\mathcal{O}_X)$ is a {\bf scheme}, if every point of $X$ has an open neighbourhood $U$ such that $(U,\mathcal{O}_{X|U})$ is an affine scheme. Its {\bf dimension} is the dimension of $X$ as a topological space
(cf.\ \ref{begin}) and it is called {\bf irreducible}, if $X$ is irreducible.
\end{definition}

\begin{definition} \label{closedsub}
A {\bf morphism} of schemes is a morphism of locally ringed spaces. Likewise, an {\bf isomorphism} of schemes is an isomorphism of locally ringed spaces. A {\bf closed subscheme} $(Z,\mathcal{O}_Z)$ of a scheme $(X,\mathcal{O}_X)$ consists of a closed subset $Z$ of $X$ and a closed immersion $(j,j^\sharp) (Z,\mathcal{O}_Z) \to (X,\mathcal{O}_X)$ where $j$ is the canonical injection. (Cf.\ \ref{morphring}.)
\end{definition}

\begin{definition}\label{sscheme}
Let $S$ be a scheme. An {\bf $S$-scheme} or a {\bf scheme over $S$} is a scheme $X$ endowed with a morphism of schemes $\pi : X \to S$. If $S = \mathrm{Spec} A$ for a ring $A$, then an $S$-scheme is also called an {\bf $A$-scheme}.
\end{definition}

\begin{example}\label{morphismglobalfield}
Let $(X,\mathcal{O}_X)$ be a scheme, for $x \in X$ let $\mathcal{O}_{X,x}$
be the stalk at $x$ (cf.\ \ref{stalk}) and let $\mathfrak{m}_x$ be its
maximal ideal. The {\bf residue field} of $x$ on $X$ is $k(x) :=
\mathcal{O}_{X,x}/\mathfrak{m}_x$. For any field $K$, by
\cite[Ex.~II.2.7]{Hartshorne:1977}, there exists a one-to-one correspondence
between morphisms from the
spectrum of $K$ to $(X,\mathcal{O}_X)$ and pairs of a point 
$x \in X$ and an inclusion $k(x) \to K$. In other words, any spectrum of a
field that contains the residue field of some point $x \in X$ as a subfield
can be considered as an $X$-scheme.
\end{example}

\begin{definition}\label{fibredproduct}
Let $S$ be a scheme and let $X$, $Y$ be $S$-schemes with respect to the
morphisms $\pi_X : X \to S$ and $\pi_Y : Y \to S$. The {\bf fibred
product} $X \times_S Y$ of $X$ and $Y$ over $S$ is defined to be an 
$S$-scheme together with morphisms $p_X : X \times_S Y \to X$ and 
$p_Y : X \times_S Y \to Y$ satisfying $\pi_X \circ p_X = \pi_Y \circ p_Y$, 
such that given any $S$-scheme $Z$ and morphisms $f : Z \to X$ and 
$g : Z \to Y$ satisfying $\pi_X \circ f = \pi_Y \circ g$, then there exists
a unique morphism $\theta : Z \to X \times_S Y$ satisfying $f = p_X \circ
\theta$ and $g = p_Y \circ \theta$.
$$\xymatrix{
Z \ar@{-->}[rrr]^\theta \ar@/^/[drr]_f \ar@/^/[drrrr]^g & & & X \times_S Y 
\ar[dl]^{p_X} \ar[dr]^{p_Y}& \\
& & X \ar[dr]^{\pi_X} & & Y \ar[dl]_{\pi_Y} \\
& & & S &
} 
$$
By \cite[II.3.3]{Hartshorne:1977}, \cite[3.1.2]{Liu:2002} the fibred product of two $S$-schemes exists and is unique
up to unique isomorphism.
\end{definition}

\begin{example} 
Let $A$ be a ring, let $B = \bigoplus_{d \geq 0} B_d$ be a graded $A$-algebra, and let $B_+ := \bigoplus_{d > 0} B_d$. Define the set $\mathrm{Proj}(B)$ to be the set of homogeneous prime ideals of $B$ which do not contain $B_+$, and define a topology on $\mathrm{Proj}(B)$ by taking the closed sets to be sets of the
form $V(\mathfrak{a}) = \{ \mathfrak{p} \in \mathrm{Proj}(B) \mid \mathfrak{a} \subseteq \mathfrak{p} \}$, for arbitrary homogeneous ideals $\mathfrak{a}$ of $A$; cf.\ \cite[II.2.4]{Hartshorne:1977}, \cite[p.~50]{Liu:2002}. 
For each $\mathfrak{p} \in \mathrm{Proj}(B)$ let $B_\mathfrak{p}$ be the ring of elements of degree zero in the localized ring $T^{-1}B$, where $T$ is the multiplicative system of all homogeneous elements of $B$ which are not in $\mathfrak{p}$.
For an open set $U \subseteq \mathrm{Proj}(B)$ define $\mathcal{O}(U)$ to be the set of functions $s : U \to \bigsqcup_{\mathfrak{p} \in U} B_\mathfrak{p}$ such that $s(\mathfrak{p}) \in B_\mathfrak{p}$ and for each $\mathfrak{p} \in U$ there exist a neighbourhood $V \subseteq U$ of $\mathfrak{p}$ and homogeneous elements $a, f \in B$ of the same degree with $f \not\in \mathfrak{q}$ for each $\mathfrak{q} \in V$ and $s(\mathfrak{q}) = a/f$ in $B_\mathfrak{q}$. The resulting ringed space $\mathrm{Proj} B = (\mathrm{Proj}(B),\mathcal{O})$ is a scheme, in fact an $A$-scheme, cf.\ \cite[II.2.5]{Hartshorne:1977}, \cite[2.3.38]{Liu:2002}.
\end{example}

\begin{example}\label{pna}
Let $A$ be a ring and let $B = A[x_0,x_1,...,x_n]$ with grading by degree. Then the scheme $\mathbb{P}_n(A) := \mathrm{Proj} B$ is the {\bf projective $n$-space over $A$}.
\end{example}

\newparagraph \label{catschemes}\label{duality}
The schemes together with their morphisms form a category $\mathfrak{Sch}$. 
It contains the category $\mathfrak{AffSch}$ of affine schemes as a 
subcategory. The functor $\mathrm{Spec}$ from $\mathfrak{CRing}$ to 
$\mathfrak{AffSch}$ that sends $A$ to $\mathrm{Spec} A$ and a ring 
homomorphism $A \to B$ to the corresponding morphism of ringed spaces from 
the spectrum of $B$ to the spectrum of $A$ 
(cf.~\cite[II.2.3]{Hartshorne:1977}) is a contravariant equivalence (also 
called a duality) between the category of commutative rings and the category 
of affine schemes. The global section functor, i.e., the functor from
$\mathfrak{AffSch}$ to $\mathfrak{CRing}$ that sends 
$(\mathrm{Spec}(A),\mathcal{O})$ to $\mathcal{O}(\mathrm{Spec}(A))$, is an
inverse functor to $\mathrm{Spec}$; cf.\ 
\cite[II.2.2, II.2.3]{Hartshorne:1977}.  

\medskip
\newparagraph \label{finalobject}
The affine scheme $\mathrm{Spec} \mathbb{Z}$ is a final
object in both categories of schemes and affine schemes; cf.\ 
\cite[Ex.~II.2.5]{Hartshorne:1977}.

\begin{definition}\label{curvesch}\label{nonsingular2}
Let $k$ be a perfect field. A {\bf projective scheme over $k$} is a scheme over $k$ that is isomorphic to a closed subscheme of the scheme $\mathbb{P}_n(k)$, for some $n$ (cf.\ \ref{closedsub}, \ref{pna}).  
A {\bf projective curve over $k$} is an irreducible projective scheme of dimension $1$.
A projective curve $Y$ is called {\bf non-singular} at $P \in Y$, if $\mathcal{O}_P$ (cf.\ \ref{stalk}) is a discrete valuation ring (cf.\ \cite[4.2.9]{Liu:2002}). A projective curve is called {\bf non-singular}, if it is non-singular at each of its points. 
\end{definition}

\newparagraph \label{vectormodule}
Comparing \ref{nonsingular} with \ref{nonsingular2} it is clear that a 
non-singular projective curve, defined as a variety, yields a non-singular
projective curve if considered as a scheme. 
In the context of \ref{curveasscheme} one can convert locally free $\mathcal{O}_Y$-modules (\ref{oxmodule}) into vector bundles over $Y$ (\ref{vectorbundle}) and vice versa, cf.\ \cite[p.~61]{Harder:2008}, \cite[Ex.\ II.5.18]{Hartshorne:1977}, \cite{Serre:1955}:
a vector bundle $\pi : E \to Y$ yields a sheaf $\mathcal{L}_E$ by sending an open subset $U$ of $Y$ to the (fibre-wise) $\mathcal{O}_Y(U)$-module of sections $U \to E$ over $U$.

\medskip
\newparagraph
 Let $A$ be a ring. The quotient of $A[x_0,x_1,...,x_n]$ by a homogeneous ideal  is a {\bf homogeneous $A$-algebra}. By \cite[2.3.41]{Liu:2002}, for every homogeneous $A$-algebra $B$, the scheme $\mathrm{Proj} B$ is a projective scheme and, by \cite[5.1.30]{Liu:2002}, conversely every projective scheme over $A$ is isomorphic to $\mathrm{Proj} B$, for some homogeneous $A$-algebra $B$.  

\begin{definition}
Let $X$ be a scheme. For each open subset $U$, let $S(U)$ denote the set of elements $s \in \mathcal{O}_X(U)$ such that for each $x \in U$ the germ $s_x$ is not a zero divisor in the stalk $\mathcal{O}_x$. Then the {\bf sheaf $\mathcal{K}$ of total quotient rings} of the sheaf $\mathcal{O}_X$ is the sheaf associated to the presheaf $U \mapsto S(U)^{-1}\mathcal{O}_X(U)$ (cf.\ 
\ref{sheafification}). Let $\mathcal{K}^*$ and $\mathcal{O}^*$ denote the sheaves of groups of invertible elements in the sheaves of rings $\mathcal{K}$, resp.\ $\mathcal{O}_X$. 
An element of the group $(\mathcal{K}^*/\mathcal{O}^*)(X)$ is called a {\bf Cartier divisor}. A Cartier divisor is {\bf principal} if it is in the image of the natural map $\mathcal{K}^*(X) \to (\mathcal{K}^*/\mathcal{O}^*)(X)$. Two Cartier divisors are {\bf linearly equivalent} if their difference is principal.
\end{definition}

\begin{proposition}[{\cite[II.6.11]{Hartshorne:1977}, \cite[7.2.16]{Liu:2002}}] \label{weilcartier}
Let $Y$ be a non-singular projective curve over $\mathbb{F}_q$. Then there 
exists a natural isomorphism between the group $\mathrm{Div}(Y)$ of Weil 
divisors (cf.~\ref{divisor}) and the group $\mathcal{K}^*/\mathcal{O}^*(Y)$ of Cartier divisors and, furthermore, the principal Weil divisors correspond to the principal Cartier divisors under this isomorphism.
\end{proposition}

\newparagraph \label{concreteweil}
In fact, as $\mathcal{O}_Y$ is integral (\cite[2.4.17]{Liu:2002}), the 
sheaf $\mathcal{K}$ is the constant sheaf corresponding to the function 
field $K$ of $Y$. As in \ref{quotient}, a Cartier divisor is given by a 
family $\{ U_i, f_i\}$ where $\{U_i\}$ is an open cover of $Y$ and 
$f_i \in \mathcal{K}^*(U_i)= K^*$. For each $P \in Y$ and for all $i$, $j$
such that $P \in U_i, U_j$ one has $\nu_P(f_i) = \nu_P(f_j)$, as
$\frac{f_i}{f_j}$ is invertible on $U_i \cap U_j$. 
One therefore obtains the well-defined Weil divisor 
$\sum_{P \in Y^\circ} \nu_P(f_i)P$. 

\begin{definition}
Let $D$ be a Cartier divisor on a non-singular projective curve $Y/\mathbb{F}_q$ represented by a family $\{ U_i, f_i\}$ where $\{U_i\}$ is an open cover of $Y$ and $f_i \in \mathcal{K}^*(U_i)$. The {\bf sheaf $\mathcal{L}(D)$ associated to $D$} is the subsheaf of $\mathcal{K}$ given by taking $\mathcal{L}(D)$ to be the sub-$\mathcal{O}_Y$-module of $\mathcal{K}$ generated by $f_i^{-1}$ on $U_i$. Since $\frac{f_i}{f_j}$ is invertible on $U_i \cap U_j$, the elements $f_i^{-1}$ and $f_j^{-1}$ generate the same $\mathcal{O}_Y$-module, and hence $\mathcal{L}(D)$ is well defined.
\end{definition}

\begin{proposition}[{\cite[II.6.13, II.6.15]{Hartshorne:1977}, \cite[7.1.19]{Liu:2002}}] \label{cartierinvert}
Let $Y$ be a non-singular projective curve over $\mathbb{F}_q$. Then for any Cartier divisor $D$, the sheaf $\mathcal{L}(D)$ is an invertible $\mathcal{O}_Y$-module. Moreover, the map $D \mapsto \mathcal{L}(D)$ induces a surjection from the group of Cartier divisors onto the Picard group $\mathrm{Pic}(Y)$. The kernel of this surjection is the group of principal Cartier divisors. 
\end{proposition}

\begin{definition} \label{c}
Let $Y/\mathbb{F}_q$ be a non-singular projective curve. Define the {\bf degree} $c(\mathcal{L})$ of an invertible $\mathcal{O}_Y$-module $\mathcal{\mathcal{L}}$ to be the degree $\mathrm{deg}(D)$ of the corresponding Weil divisor; cf.\ \ref{weilcartier}, \ref{concreteweil}, \ref{cartierinvert}; \cite[7.3.1, 7.3.2]{Liu:2002}.
\end{definition}

\begin{example}
Let $Y = \mathbb{P}_1 = \{ (x:y) \}$ and let $U_1 = \{ (x:y) \mid y \neq 0 \}$
and $U_2 = \{ (x:y) \mid x \neq 0 \}$ be an open cover of $Y$ with local 
coordinates $s = \frac{x}{y}$ and $t = \frac{y}{x}$. For $d \in \mathbb{N}$ 
define the line bundle $L(d) = \{ ((x:y),(a,b)) \in \mathbb{P}_1 \times A_2 
\mid x^db = y^da \}$ with canonical projection $\pi : L(d) \to \mathbb{P}_1$. Then 
\begin{eqnarray*}
U_1 \times A_1 & \stackrel{f_1}{\cong} &  \pi^{-1}(U_1) = \{ ((x:y),(a,b)) \in U_1 \times A_2 \mid a = s^db \} \\
(s,b) & \mapsto & ((s:1),(s^db,b)) \\
 U_2 \times A_1 & \stackrel{f_2}{\cong} & \pi^{-1}(U_2) = \{ ((x:y),(a,b)) \in U_1 \times A_2 \mid b = t^da \} \\
(t,a) & \mapsto & ((1:t),(a,t^da)).
\end{eqnarray*}
One has ${f_1}_{|U_1 \cap U_2} = s^d {f_2}_{|U_1 \cap U_2}$, so that the line bundle $L(d)$ corresponds to the Weil divisor $\sum_{P \in Y^\circ} \nu_P(f_i)P$ of degree $d$ (cf.~\ref{concreteweil}).
\end{example}

\begin{theorem}[Riemann--Roch Theorem, {\cite[IV.1.3]{Hartshorne:1977}, \cite[7.3.26]{Liu:2002}}] \label{RiemannRoch2}
Let $Y/\mathbb{F}_q$ be a non-singular projective curve of genus $g$. Then for any invertible $\mathcal{O}_Y$-module $\mathcal{L}$ $$\mathrm{dim}_{\mathbb{F}_q}(H^0(Y,\mathcal{L})) - \mathrm{dim}_{\mathbb{F}_q}(H^1(Y,\mathcal{L})) = c(\mathcal{L}) + 1 - g.$$
\end{theorem}
Here $H^i(Y,\mathcal{L})$ denotes sheaf cohomology groups as defined in 
\cite[III]{Hartshorne:1977}, \cite[5.2]{Liu:2002}.

\medskip
\newparagraph \label{RiemannRoch3}
For a locally free $\mathcal{O}_Y$-module $\mathcal{E}$ of rank $n$ define $c(\mathcal{E}) := \sum_{i=1}^n c(\mathcal{E}_i/\mathcal{E}_{i-1})$ where $0 \subset \mathcal{E}_1 \subset \cdots \subset \mathcal{E}_{n-1} \subset \mathcal{E}_n = \mathcal{E}$ is a filtration such that each $\mathcal{E}_i/\mathcal{E}_{i-1}$ is invertible; cf.\ \cite[2.1]{Grothendieck:1957}, \cite[p.~122]{Harder:1968}. Then \ref{RiemannRoch2} implies
$\mathrm{dim}_{\mathbb{F}_q}(H^0(Y,\mathcal{E})) - \mathrm{dim}_{\mathbb{F}_q}(H^1(Y,\mathcal{E})) = c(\mathcal{E}) + (1 - g)n$.
(Cf.\ \cite[p.~99]{Weil:1995}.)

\section{Reduction theory for rationally trivial group schemes} \label{reductiontheoryrtgs}

\begin{introduction}
In this section I describe the heart of Harder's reduction theory based on
\cite{Harder:1968}. From this section on I will assume the reader has a basic intuition for
linear algebraic groups and feels comfortable with some of the standard 
terminology such as that of Borel subgroups and parabolic subgroups. For
introductory reading the sources \cite{Borel:1991} and \cite{Humphreys:1975},
for further reading the sources
\cite{Jantzen:2003}, \cite{Springer:1998}, \cite{Voskresenski:1998}, 
\cite{Waterhouse:1979} are highly recommended. I also recommend \cite[Appendix E]{Laumon:1997}.
\end{introduction}

\begin{definition}
An object $X$ of a category $\mathfrak{C}$ is called a {\bf group object}, 
if for each object $Y$ of $\mathfrak{C}$ the set 
$\mathrm{Hom}_\mathfrak{C}(Y,X)$ of morphisms from $Y$ to $X$ is a group and the
correspondence $Y \mapsto \mathrm{Hom}_\mathfrak{C}(Y,X)$ is a 
(contravariant) functor from the category $\mathfrak{C}$ to the category 
$\mathfrak{Gr}$ of groups; cf.\ \cite[p.~3]{Voskresenski:1998}. 
\end{definition}

\begin{definition}
A group object in the categories $\mathfrak{Sch}$ of schemes, resp.\ 
$\mathfrak{AffSch}$ of affines schemes (cf.\ \ref{catschemes}) is called a
{\bf group scheme}, resp.\ an {\bf affine group scheme}.
\end{definition}

\newparagraph\label{affinegroupschemeisgroupscheme}
As the categories $\mathfrak{Sch}$ and $\mathfrak{AffSch}$ both have $\mathrm{Spec} \mathbb{Z}$ as a final object (cf.\ \ref{finalobject}) and
have finite products, being an (affine) group scheme is an intrinsic
property; cf.\ \cite[p.~3]{Voskresenski:1998}. In particular, any affine group scheme is a group scheme.  

\medskip
\newparagraph \label{duality2}
When restricting (and co-restricting accordingly) the functor $\mathrm{Spec}$ 
from $\mathfrak{CRing}$ to $\mathfrak{AffSch}$ (cf.~\ref{catschemes}) to the 
subcategory of commutative Hopf algebras, it yields a contravariant 
equivalence (duality) to the category of affine group schemes; cf.\ \cite[1.3]{Voskresenski:1998}, \cite[p.~9]{Waterhouse:1979}.

\begin{example}
Let $\mathrm{G}_\mathrm{a}$ be the covariant functor on the category $\mathfrak{CRing}$
defined by $\mathrm{G}_\mathrm{a}(A) = (A,+)$. Since by means of duality (cf.\
\ref{duality})
$$\mathrm{Hom}_{\mathfrak{CRing}}(\mathbb{Z}[T],A) = A \quad \Longleftrightarrow 
\quad \mathrm{Hom}_{\mathfrak{AffSch}}(\mathrm{Spec}A, \mathrm{Spec}\mathbb{Z}[T]) =
A,$$ the functor $G_\mathrm{a}$, considered as the contravariant functor
on the category $\mathfrak{AffSch}$ that sends
$(\mathrm{Spec}(A),\mathcal{O})$ to $(\mathcal{O}(\mathrm{Spec}(A)),+)$, is 
represented by $\mathrm{Spec}\mathbb{Z}[T]$. Therefore, $\mathrm{Spec}\mathbb{Z}[T]$ is an affine group scheme. By 
\ref{affinegroupschemeisgroupscheme}, it is also a group scheme, which as a
functor sends
$(X,\mathcal{O}_X)$ to $(\mathcal{O}_X(X),+)$. 
\end{example}

\begin{example}
Let $\mathrm{G}_\mathrm{m}$ be the covariant functor on the category $\mathfrak{CRing}$
defined by $\mathrm{G}_\mathrm{m}(A) = (A^*,\cdot)$. Since
$$\mathrm{Hom}_{\mathfrak{CRing}}(\mathbb{Z}[T,T^{-1}],A) = A^* \quad 
\Longleftrightarrow 
\quad \mathrm{Hom}_{\mathfrak{AffSch}}(\mathrm{Spec}A, 
\mathrm{Spec}\mathbb{Z}[T,T^{-1}]) =
A^*,$$ the functor $\mathrm{G}_\mathrm{m}$, considered as a contravariant functor
on the category $\mathfrak{AffSch}$, is represented by
$\mathrm{Spec}\mathbb{Z}[T,T^{-1}]$. Therefore, $\mathrm{Spec}\mathbb{Z}[T,T^{-1}]$ is an affine group scheme and, by 
\ref{affinegroupschemeisgroupscheme}, also a group scheme. 
\end{example}

\begin{example}
\label{gln}
Let $\mathrm{GL}_n$ be the covariant functor on the category $\mathfrak{CRing}$
defined by $$A \mapsto \mathrm{GL}_n(A) = \{ M \in A^{n \times n} \mid \mathrm{det}(M)
\neq 0 \}.$$ It is represented by the affine scheme $\mathrm{GL}_n = \mathrm{Spec}
\mathbb{Z}[T_{11}, ..., T_{nn}, \mathrm{det}(T_{ij})^{-1}]$ and, hence, is
an affine group scheme. By
\ref{affinegroupschemeisgroupscheme}, $\mathrm{GL}_n$ is also a group scheme, and
induces the contravariant functor on $\mathfrak{Sch}$ that sends 
$(X,\mathcal{O}_X)$ to $\mathrm{GL}_n(\mathcal{O}_X(X))$.

An alternative, basis-free way of defining $\mathrm{GL}_n$ is to define a 
contravariant functor on $\mathfrak{Sch}$ that sends 
$$(X,\mathcal{O}_X) \quad \mbox{to} \quad
\mathrm{Aut}_{\mathcal{O}_X(X)}\left(\bigoplus_{i=1}^n
\mathcal{O}_X(X)\right).$$
\end{example}

\begin{example} \label{sln}
Let $\mathrm{SL}_n$ be the covariant functor on the category $\mathfrak{CRing}$
defined by $$A \mapsto \mathrm{ker}\left(\mathrm{GL}_n(A) \to
\mathrm{G}_{\mathrm{m}}(A) \right) = \{ M \in A^{n \times n} \mid \mathrm{det}(M)
= 1 \}.$$ By Yoneda's lemma (\cite[p.~6]{Waterhouse:1979}) the natural 
transformation $\mathrm{GL}_n \to
\mathrm{G}_{\mathrm{m}}$ can be described by a homomorphism 
$\mathbb{Z}[T,T^{-1}] \to 
\mathbb{Z}[T_{11}, ..., T_{nn}, \mathrm{det}(T_{ij})^{-1}]$, the one that
sends $T$ to $\mathrm{det}(T_{ij})$. Therefore $\mathrm{SL}_n$ is
represented by $$\mathbb{Z}[T_{11}, ..., T_{nn}, \mathrm{det}(T_{ij})^{-1}]
\otimes_{\mathbb{Z}[T,T^{-1}]} \mathbb{Z} = 
\mathbb{Z}[T_{11}, ..., T_{nn}]/(\mathrm{det}(T_{ij})-1).$$ In other words,
$\mathrm{Spec} 
\mathbb{Z}[T_{11}, ..., T_{nn}]/(\mathrm{det}(T_{ij})-1)$ is an (affine)
group scheme.
\end{example}

\begin{definition} \label{sgroupscheme}
An $S$-scheme (cf.\ \ref{sscheme}) that is an (affine) group scheme is
called an {\bf (affine) group $S$-scheme}. As the categories of (affine)
$S$-schemes have finite products and have the scheme $S$ as a final object,
as in \ref{affinegroupschemeisgroupscheme} being an (affine) group
$S$-scheme is an intrinsic property. In particular, any affine group
$S$-scheme is a group $S$-scheme.
\end{definition}

\begin{example} \label{slngen}
Generalizing \ref{sln}, let $S$ be a scheme, let $G$ and $H$ be group $S$-schemes, and let 
$f : G \to H$ be a homomorphism of group $S$-schemes, i.e., a morphism of
$S$-schemes from $G$ to $H$ such that for each $S$-scheme $Y$ the induced
map $\mathrm{Hom}_\mathfrak{C}(Y,G) \to \mathrm{Hom}_\mathfrak{C}(Y,H)$ is a
group homomorphism. Define a functor on the category of $S$-schemes by sending
$$Y \quad \mbox{to} \quad 
\mathrm{ker}(\mathrm{Hom}_\mathfrak{C}(Y,G) \to
\mathrm{Hom}_\mathfrak{C}(Y,H)).$$
It is represented by the fibred product $S$-scheme $G \times_{H} S$,
cf.~\ref{fibredproduct}.
\end{example}

\begin{example} \label{sle}
Let $(X,\mathcal{O}_X)$ be an (irreducible) scheme and let $\mathcal{E}$ be
a locally free $\mathcal{O}_X$-module of finite rank (\ref{oxmodule}).
Following \cite[I~4.5]{Demazure/Grothendieck:1970a} define a functor $\mathrm{GL}(\mathcal{E})$ on the category of $X$-schemes that sends
$$(Y,\mathcal{O}_Y) \quad \mbox{to} \quad
\mathrm{Aut}_{\mathcal{O}_Y(Y)}\left(\left(\mathcal{O}_Y
\otimes_{f^{-1}\mathcal{O}_X} f^{-1}\mathcal{E}\right)(Y)\right),$$ 
where $f$ denotes the morphism of schemes from the $X$-scheme 
$(Y,\mathcal{O}_Y)$ to the scheme $(X,\mathcal{O}_X)$ (\ref{sscheme}) and $f^{-1}\mathcal{O}_X$ and $f^{-1}\mathcal{E}$ denote the respective inverse image sheaves (\ref{directimage}).
For $\mathcal{E} = \bigoplus_{i = 1}^n \mathcal{O}_X$ this is just the affine group
$X$-scheme
$\mathrm{GL}_n/X$ from \ref{gln}. The present more general example is obtained from that
example by a twisting process; cf.\ \cite[p.~134]{Harder:1968}, 
\cite[p.~117]{Hartshorne:1977}. In analogy to \ref{sln} define the (affine)
group $X$-scheme $\mathrm{SL}(\mathcal{E})$ as 
$\mathrm{ker}(\mathrm{GL}(\mathcal{E}) \to \mathrm{G}_\mathrm{m}/X)$
(cf.~\ref{slngen}).
\end{example}

\begin{definition}
Let $Y/\mathbb{F}_q$ be a non-singular projective curve (\ref{vectormodule}), let $K =
\mathbb{F}_q(Y)$ be its
field of $\mathbb{F}_q$-rational functions (\ref{rationalfunctions})
considered as the residue field 
of $Y$ at its generic point (\ref{morphismglobalfield} and
\cite[Ex.~II.3.6]{Hartshorne:1977}), let $G/Y$ be an
affine 
group $Y$-scheme (\ref{sgroupscheme}), and consider $\mathrm{Spec}K$ as a
$Y$-scheme via the identity map on $K$ (\ref{morphismglobalfield}). The scheme $G/Y$ is called a {\bf rationally trivial group
($Y$-)scheme} if $G/K := G \times_Y \mathrm{Spec}K$ (\ref{fibredproduct}) is a Chevalley scheme
(\cite[I.8.3.1]{Grothendieck:1960}). It is called {\bf reductive}, if $G/K = G
\times_Y \mathrm{Spec} K$ is reductive
(\cite[XIX~1.6]{Demazure/Grothendieck:1970c}). 
\end{definition}

\begin{proposition}[{\cite[XX~1.3,~1.5,~1.16,~1.17; XXII~5.6.5,~5.9.5; XXVI~1.12,~2.1]{Demazure/Grothendieck:1970c}}] \label{filtration}
Let $Y/\mathbb{F}_q$ be a non-singular projective curve, let $G/Y$ be a 
reductive group $Y$-scheme, and let $P/Y$ be a parabolic subgroup of $G/Y$
(cf.~\cite[XXVI~1.1]{Demazure/Grothendieck:1970c}). Then the unipotent 
radical $R_u(P)$ (cf.~\cite[XIX~1.2]{Demazure/Grothendieck:1970c}) admits a 
filtration $R_u(P) = U_0 \supset U_1 \supset \cdots \supset U_k = \{ e \}$ 
such that each $U_i/U_{i+1}$ is a vector bundle over $Y$. More precisely, if 
$P_\alpha$ denotes the vector bundle over $Y$ corresponding to the root space 
$\mathfrak{g}^\alpha$ (cf.~\cite[XIX~1.10]{Demazure/Grothendieck:1970c}), then $$U_i = \prod_{\alpha \in \Delta^+_P, l(\alpha)>i} P_\alpha \quad \quad \mbox{ and } \quad \quad U_i/U_{i+1} \cong \prod_{\alpha \in \Delta^+_P, l(\alpha)=i+1} P_\alpha.$$ 
\end{proposition}

\begin{example} \label{borelcorrespondence}
Let $A_1 \cong \mathcal{O}_Y \cong A_2$, let $\mathcal{E} = A_1 \oplus A_2$ 
and let $\mathrm{SL}(\mathcal{E}) = \mathrm{SL}_2/Y$ be the group $Y$-scheme defined in
\ref{sln}, \ref{sle}. Each of $\mathcal{H}om_{\mathcal{O}_Y}(A_1,A_2)$ and 
$\mathcal{H}om_{\mathcal{O}_Y}(A_2,A_1)$ equals the unipotent 
radical of a Borel subgroup of $\mathrm{SL}(A_1 \oplus A_2)$ and is an
invertible $\mathcal{O}_Y$-module, because 
$\mathcal{H}om_{\mathcal{O}_Y}(A_i,A_j) \cong A_i^\vee \otimes_{\mathcal{O}_Y} 
A_j$. By \ref{vectormodule}, this corresponds to a vector bundle 
(of dimension $1$), confirming \ref{filtration} for $\mathrm{SL}_2$.

More generally, let $\mathcal{E}$ be a locally free $\mathcal{O}_Y$-module 
of rank $2$ with an invertible $\mathcal{O}_Y$-submodule $\mathcal{L}$ such 
that $\mathcal{E}/\mathcal{L}$ is also invertible and let 
$\mathrm{SL}(\mathcal{E})$ be the group
$Y$-scheme defined in \ref{sle}. 
There is a one-to-one correspondence between the Borel subgroups $B$ of $\mathrm{SL}(\mathcal{E})$ and the invertible $\mathcal{O}_Y$-submodules $\mathcal{L}$ such that $\mathcal{E}/\mathcal{L}$ is also invertible, given by $\mathcal{L} \mapsto \mathrm{Stab}_{\mathrm{SL}(\mathcal{E})}(\mathcal{L})$. The unipotent radical $R_u(B)$ is isomorphic to $\mathcal{H}om_{\mathcal{O}_Y}(\mathcal{E}/\mathcal{L},\mathcal{L})$.
For any invertible $\mathcal{O}_Y$-module $\mathcal{H}$, the equality $\mathcal{H}om_{\mathcal{O}_Y}(\mathcal{E}/\mathcal{L} \otimes_{\mathcal{O}_Y} \mathcal{H}, \mathcal{L} \otimes_{\mathcal{O}_Y} \mathcal{H}) = (\mathcal{E}/\mathcal{L} \otimes_{\mathcal{O}_Y} \mathcal{H})^\vee \otimes_{\mathcal{O}_Y} (\mathcal{L} \otimes_{\mathcal{O}_Y} \mathcal{H}) = \mathcal{H}om_{\mathcal{O}_Y}(\mathcal{E}/\mathcal{L},\mathcal{L})$ implies $\mathrm{SL}(\mathcal{E} \otimes_{\mathcal{O}_Y} \mathcal{H}) \cong \mathrm{SL}(\mathcal{E})$. (Cf.~\cite[I, 4.5]{Demazure/Grothendieck:1970a},\cite[XX,
5.1]{Demazure/Grothendieck:1970c}.)
\end{example}

\begin{definition}
Let $G/Y$ be a rationally trivial group $Y$-scheme, let $B/Y$ be a Borel subgroup of $G/Y$ and let $\{ \alpha_1, ..., \alpha_r \}$ be the simple roots of $B$. Using the notation introduced in
\ref{vectormodule}, \ref{c} and \ref{filtration}, define $$n_i(B) := c\left(\mathcal{L}_{B_{\alpha_i}}\right).$$
Note that, since $G/Y$ is rationally trivial, each $B_{\alpha_i}$ is a vector bundle of dimension $1$, so that $\mathcal{L}_{B_{\alpha_i}}$ is an invertible $\mathcal{O}_Y$-module.
\end{definition}

\begin{theorem}[{\cite[2.2.6]{Harder:1968}}]\label{reduction1}
Let $G/Y$ be a rationally trivial group $Y$-scheme of genus $g$ (cf.~\ref{genus}). Then there exists a Borel subgroup $B/Y$ such that $n_i(B) \geq -2g$ for all $i \in I = \{ 1, ..., r \}$.
\end{theorem}

\begin{proof}
In case $G/Y$ is of type $A_1$, i.e., if there exists a locally free $\mathcal{O}_Y$-module $\mathcal{E}$ of rank $2$ such that $G \cong \mathrm{SL}(\mathcal{E})$, one can proceed as follows. Let $\mathcal{L}$ be an invertible $\mathcal{O}_Y$-submodule such that $\mathcal{E}/\mathcal{L}$ is also invertible and let $B = \mathrm{Stab}_{\mathrm{SL}(\mathcal{E})}(\mathcal{L})$ (cf.\ \ref{borelcorrespondence}).

Then $n_1(B) = c(R_u(B)) = c(\mathcal{H}om_{\mathcal{O}_Y}(\mathcal{E}/ \mathcal{L}, \mathcal{L})) = c((\mathcal{E}/ \mathcal{L})^\vee \otimes_{\mathcal{O}_Y} \mathcal{L}) = c(\mathcal{L}) - c(\mathcal{E} / \mathcal{L}) = 2c(\mathcal{L}) - c(\mathcal{E})$.
By \ref{invertible1} and \ref{c} there exists an invertible $\mathcal{O}_Y$-module $\mathcal{H}$ with $c(\mathcal{H})= 1$. Since $\mathrm{SL}(\mathcal{E} \otimes_{\mathcal{O}_Y} \mathcal{H}^\vee) \cong \mathrm{SL}(\mathcal{E}) \cong \mathrm{SL}(\mathcal{E} \otimes_{\mathcal{O}_Y} \mathcal{H})$ by \ref{borelcorrespondence}, the formula $c(\mathcal{E} \otimes_{\mathcal{O}_Y} \mathcal{H}^\vee) = c(\mathcal{E}) - 2 c(\mathcal{H}) = c(\mathcal{E}) - 2$ allows us to assume without loss of generality that $2g-2 < c(\mathcal{E}) \leq 2g$.

By the Riemann--Roch theorem (cf.\ \ref{RiemannRoch3}) there exists 
$0 \neq t \in H^0(Y,\mathcal{E}) \cong \mathcal{E}(Y)$. Since 
$\mathcal{E}$ is locally free of rank $2$ there exists an open cover 
$\{ U_i\}$ of $Y$ such that $\mathcal{E}_{|U_i} \cong {\mathcal{O}_Y}_{|U_i} 
\oplus {\mathcal{O}_Y}_{|U_i}$ for each $U_i$. 
Therefore $t$ is contained in an invertible $\mathcal{O}_Y$-submodule 
$\mathcal{L}_t$ such that $\mathcal{E} / \mathcal{L}_t$ is also invertible. 
As $t \in H^0(Y,\mathcal{L}_t)$, one has $c(\mathcal{L}_t) \geq 0$ by \cite[Lemma IV.1.2]{Hartshorne:1977}. Hence $$n_1(B_t) = c(\mathcal{H}om_{\mathcal{O}_Y}(\mathcal{E}/ \mathcal{L}_t, \mathcal{L}_t)) = 2c(\mathcal{L}_t) - c(\mathcal{E}) \geq -2g,$$
where $B_t = \mathrm{Stab}_{\mathrm{SL}(\mathcal{E})}(\mathcal{L}_t)$ (cf.\ \ref{borelcorrespondence}). 

The general case can be reduced to the case $A_1$ via a local to global argument.
\end{proof}

\begin{definition}
Let $c_1 < -2g$. A Borel subgroup $B/Y$ of $G/Y$ is {\bf reduced}, if $n_i(B) \geq c_1$ for all $i$.
\end{definition}

\begin{theorem}[{\cite[2.2.13, 2.2.14]{Harder:1968}}] \label{reduction2}
Let $G/Y$ be a rationally trivial group $Y$-scheme. Then there exist 
constants $c_2 > \gamma > c_1$ such that the following hold: Let $B/Y$ 
be a reduced Borel subgroup of $G/Y$ and let $\alpha_{i_0}$ be a simple 
root of $B$ such that $n_{i_0}(B) \geq c_2$. Then each reduced Borel 
subgroup $B'$ of $G/Y$ satisfies $n_{i_0}(B') \geq \gamma$ and is contained
in $P_{i_0}(B)$, where $P_{i_0}(B)$ denotes the maximal parabolic of type $\alpha_{i_0}$ containing $B$. 
\end{theorem}

\begin{proposition}[{\cite[2.2.3, 2.2.11]{Harder:1968}}] \label{isolatedparabolic2}
To each reduced Borel subgroup $B/Y$ of $G/Y$, let 
$I^B := \{ i \in I \mid n_i(B) \geq c_2 \}$, and let 
$P^B := \bigcap_{i \in I^B} P_i(B)$. Then $$P := \bigcap_{B \subset G/Y \mbox{ reduced}} P^B$$ either equals $G/Y$ or is a parabolic subgroup of $G/Y$.
\end{proposition}

\begin{proof}
Let $B$ be a reduced Borel subgroup of $G$. By \ref{reduction2} each reduced 
Borel subgroup $B'$ of $G$ is contained in $P^{B}$ and conversely, again 
by \ref{reduction2}, $B \subset P^{B'}$. Hence $P$ contains a Borel group 
and therefore either equals $G/Y$ or is a parabolic subgroup of $G/Y$. 
\end{proof}

\noindent
Theorem \ref{reduction2} is the heart of Harder's reduction theory. The only proof of \ref{reduction2} known to me in fact uses \ref{isolatedparabolic2}, as can be guessed from the numbering used in \cite{Harder:1968}. However, as the length of the proof of \ref{reduction2} by far surpasses anything I can reasonably include in this survey and as I will need to refer to \ref{isolatedparabolic2} later, I took the liberty of deducing \ref{isolatedparabolic2} from \ref{reduction2} for the sake of this exposition.

\begin{proposition}[{\cite[p.~120]{Harder:1968}, \cite[p.~39]{Harder:1969}}] \label{sumofcharacters}
Let $G/Y$ be a reductive group $Y$-scheme, let $P/Y$ be a parabolic $K$-subgroup of $G/Y$, let $d:= \mathrm{dim}(R_u(P))$, and let $\Delta^+_{P}$ be the set of positive roots of $P$. Then the character $$\chi_{P} : P \stackrel{\mathrm{Ad}}{\rightarrow} \mathrm{GL}\left(\mathrm{Lie}(R_u(P))\right) \stackrel{\mathrm{det}}{\rightarrow} \mathrm{GL}\left(\bigwedge^{d} \mathrm{Lie}(R_u(P))\right) \cong \mathrm{G}_\mathrm{m},$$ considered as a character of a maximal split torus contained in $P$, is given by $\chi_{P} = \sum_{\alpha \in \Delta^+_{P}} \mathrm{dim}(P_\alpha) \alpha$.
\end{proposition}

\begin{definition} \label{ccc}
Let $B$ be a minimal parabolic $K$-subgroup of $G/Y$, let $\{ \alpha_1, ..., \alpha_r \}$ be the simple roots of $B$, and let $(P_i)_i$ be the maximal parabolic $K$-subgroups of $G$ of type $\alpha_i$. Using the notation of \ref{vectormodule}, \ref{RiemannRoch3} and \ref{filtration}, define $$p_i(B) := p(P_i) := \sum_{\alpha \in \Delta^+_{P_i}} c(\mathcal{L}_{P_\alpha}).$$ 
\end{definition}

\newparagraph\label{piandnu}
Let $B$ be a minimal parabolic $K$-subgroup of $G$, let $R(B)$ be the radical of $B$, let $R_u(B)$ be the unipotent radical of $B$, let $T = R(B)/R_u(B)$, let $S \subseteq T$ be the maximal $K$-split subtorus of $T$, let $\pi = \{ \alpha_1, ..., \alpha_r \} \subset X(S)$ be the system of simple roots, and let $X(B) = \mathrm{Hom}_K(B,\mathrm{G}_\mathrm{m})$ be the module of $K$-rational characters of $B$ so that $X(B) \otimes \mathbb{Q} = X(S) \otimes \mathbb{Q}$, let $P_i \supseteq B$ be the maximal $K$-parabolic of type $\alpha_i$, let $\chi_{P_i} : P_i \to \mathrm{G}_\mathrm{m}$ be the sum of roots of $P_i$ (cf.\ \ref{sumofcharacters}), and let $\chi_i := {\chi_{P_i}}_{|B}$.
The $\chi_i$ form a basis of $X(B) \otimes \mathbb{Q}$ and, if $(\cdot,\cdot)$ is a positive definite bilinear form on $X(B) \otimes \mathbb{Q}$ which is invariant under the action of the Weyl group, then
\begin{eqnarray}
(\chi_i,\alpha_j) & = & 0, \quad \mbox{if $i \neq j$, and} \label{orthogonal} \\
(\chi_i,\alpha_i) & > & 0 \quad \mbox{for all $i \in I$}. \label{acute}
\end{eqnarray} 
If $G$ is rationally trivial, for $m_{i,j} \in \mathbb{Z}$ such that $\chi_i = \sum_{j=1}^r m_{i,j} {\alpha_j}$,
then
 $$p_i(B) = \sum_{j=1}^r m_{i,j} n_j(B).$$
Furthermore, for $c_{i,j} \in \mathbb{Q}$ such that $\alpha_i = \sum_{j=1}^r c_{i,j} \chi_j$,
then $$n_i(B) =  \sum_{j=1}^r c_{i,j} p_j(B).$$

\section{Reduction theory for reductive groups over the ad\`eles}

\begin{introduction}
In this section I describe the interplay between Harder's reduction theory
and Weil's geometry of numbers based on \cite{Harder:1969} in order to
arrive at a geometric version of Harder's reduction theory.
\end{introduction}

\newparagraph\label{tamagawameasure}
Let $G/Y$ be a reductive group $Y$-scheme, let $P/Y$ be a parabolic $K$-subgroup of $G/Y$, and let $\omega$ be a non-trivial volume form on $R_u(P)$ defined over $K$. This volume form yields a Haar measure of $R_u(P(\hat{\mathbb{A}}_K))$ which is independent of the choice of $\omega$, i.e., for each $0 \neq \lambda \in K$ the volume forms $\omega$ and $\lambda \omega$ yield identical Haar measures; cf.\ \cite[p.~37]{Harder:1969}, \cite[2.3.1]{Weil:1982}. 

This measure differs from the classical Tamagawa measure by the factor 
$q^{(1-g)d}$ where $d = \mathrm{dim}(R_u(P))$, cf.\ 
\cite[p.~38]{Harder:1969}, \cite{Tamagawa:1966}, 
\cite[\S~14.4]{Voskresenski:1998}, \cite[2.3]{Weil:1982},
\cite[p.~113]{Weil:1995}.

\begin{theorem}[{\cite[1.3.1]{Harder:1969}}] \label{GeometryOfNumbers}
Let $G/Y$ be a reductive group $Y$-scheme, let $P/Y$ be a maximal parabolic 
$K$-subgroup of $G/Y$, let $\omega$ be a non-trivial volume form of $R_u(P)$ 
defined over $K$, and let $\mathfrak{K} := G(\hat{\mathcal{O}}_K)$,
cf.~\ref{adeles}. Then $$\int_{R_u(P(\hat{\mathbb{A}}_K)) \cap \mathfrak{K}}
\omega_{\hat{\mathbb{A}}_K} = q^{p(P)}.$$ 
\end{theorem}

\begin{proof}
For each $\alpha \in \Delta_P^+$ let 
$0 \subset P_\alpha^{(1)} \subset \cdots \subset 
P_\alpha^{(\mathrm{dim}(P_\alpha))} = P_\alpha$ be a filtration such that 
each $P^{(i)}_\alpha/P^{(i-1)}_\alpha$ is a vector bundle of dimension $1$; 
cf.\ \cite[2.1]{Grothendieck:1957}, \cite[p.~122]{Harder:1968}.
One then computes
\begin{eqnarray*}
\int_{R_u(P(\hat{\mathbb{A}}_K)) \cap \mathfrak{K}}
\omega_{\hat{\mathbb{A}}_K} & \stackrel{\ref{filtration}}{=} & 
\prod_{\alpha \in \Delta_P^+} \left(\int_{P_\alpha(\hat{\mathbb{A}}_K) 
\cap {\hat{\mathcal{O}}_K}^{\mathrm{dim}(P_\alpha)}} \omega_{\hat{\mathbb{A}}_K} \right) \\ 
& \stackrel{\ref{RiemannRoch3}}{=} & \prod_{\alpha \in \Delta_P^+} \prod_{i=1}^{\mathrm{dim}(P_\alpha)}\left(\int_{(P^{(i)}_\alpha/P^{(i-1)}_\alpha)(\hat{\mathbb{A}}_K) 
\cap {\hat{\mathcal{O}}_K}} \omega_{\hat{\mathbb{A}}_K} \right) \\ 
& 
\stackrel{\ref{mu2}}{=} & \prod_{\alpha \in \Delta_P^+} 
q^{c(\mathcal{L}_{P_\alpha})} \\ & \stackrel{\ref{ccc}}{=} & q^{p(P)}.
\end{eqnarray*}
\end{proof}

\begin{theorem}[{\cite[1.3.2]{Harder:1969}}] \label{transformationformula}
Let $G/Y$ be a reductive group $Y$-scheme, let $P/Y$ be a maximal parabolic 
$K$-subgroup, and let $\mathfrak{K} := G(\hat{\mathcal{O}}_K)$. Then, for each 
$x \in P(\hat{\mathbb{A}}_K)$, one has $$\int_{R_u(P(\hat{\mathbb{A}}_K)) \cap \mathfrak{K}} \omega_{\hat{\mathbb{A}}_K} = |\chi_P(x)| \int_{R_u(P(\hat{\mathbb{A}}_K)) \cap {}^x\mathfrak{K}} \omega_{\hat{\mathbb{A}}_K}.$$
\end{theorem}

\begin{proof}
The absolute value of the determinant of the derivative of conjugation by $x$
\begin{eqnarray*}
|\chi_{P}(\cdot)| : P(\hat{\mathbb{A}}_K) & \stackrel{\mathrm{Ad}}{\rightarrow} & \mathrm{GL}\left(\mathrm{Lie}(R_u(P(\hat{\mathbb{A}}_K)))\right) \stackrel{\mathrm{det}}{\rightarrow} \mathrm{GL}\left(\bigwedge^{d} \mathrm{Lie}(R_u(P(\hat{\mathbb{A}}_K)))\right) \stackrel{\mathrm{|\cdot|}}{\rightarrow} \mathbb{R} \\
x & \mapsto & |\chi_P(x)|
\end{eqnarray*}
(cf.\ \ref{sumofcharacters}) measures the ratio of the volumes of 
$R_u(P(\hat{\mathbb{A}}_K)) \cap \mathfrak{K}$ and of 
$R_u(P(\hat{\mathbb{A}}_K)) \cap {}^x\mathfrak{K}$.
\end{proof}

\begin{definition} \label{piandnu2}
Using the notation of \ref{piandnu}, let 
\begin{eqnarray*}
\pi_i(B,{}^x\mathfrak{K}) := \pi(P_i,{}^x\mathfrak{K})& := & \int_{R_u(P_i(\hat{\mathbb{A}}_K)) \cap {}^x\mathfrak{K}} \omega_{\hat{\mathbb{A}}_K}
\quad \mbox{ and} \\
\nu_i(B,{}^x\mathfrak{K}) & := & \prod_{j=1}^r \pi_j(B,{}^x\mathfrak{K})^{c_{i,j}}.
\end{eqnarray*}
\end{definition}

\begin{corollary}[{\cite[p.~40]{Harder:1969}}]\label{transformationformulafornu}For each $x \in B(\hat{\mathbb{A}}_K)$, one has $\nu_i(B,\mathfrak{K}) = |\alpha_i(x)| \nu_i(B,{}^x\mathfrak{K})$.
\end{corollary}

\begin{observation} \label{kinvariance}
For each $x \in G(K)$, one has $\pi_i(B,\mathfrak{K}) = \pi_i({}^xB,{}^x\mathfrak{K})$ and $\nu_i(B,\mathfrak{K}) = \nu_i({}^xB,{}^x\mathfrak{K})$.
\end{observation}

\begin{proof}
Conjugation 
by $x \in G(K)$ maps the $K$-volume form $\omega$ on $R_u(P)$ onto a 
$K$-volume form ${}^x\omega$ on $R_u({}^xP)$, whence $$\pi_i(B,\mathfrak{K})  =  \int_{R_u(P_i(\hat{\mathbb{A}}_K)) \cap \mathfrak{K}} \omega_{\hat{\mathbb{A}}_K} = \int_{R_u({}^xP_i(\hat{\mathbb{A}}_K)) \cap {}^x\mathfrak{K}} {}^x\omega_{\hat{\mathbb{A}}_K} \stackrel{\ref{tamagawameasure}}{=} \pi_i({}^xB,{}^x\mathfrak{K}).$$
By \ref{piandnu2} the second identity follows from the first.
\end{proof}

\begin{theorem}[{\cite[2.1.1]{Harder:1969}}] \label{reduction3}
Let $G/Y$ be a rationally trivial group $Y$-scheme and let $\mathfrak{K} := G(\hat{\mathcal{O}}_K)$. Then there exists a constant $C_1 > 0$ such that for each $x \in G(\hat{\mathbb{A}}_K)$ there exists a Borel subgroup
$B/Y$ of $G/Y$ such that $$\nu_i(B,{}^x\mathfrak{K}) \geq C_1.$$
\end{theorem}

\begin{proof}
By \cite[1.1.2]{Harder:1969}, for each $x \in G(\hat{\mathbb{A}}_K)$ there 
exists a rationally trivial group scheme $G^{(x)}/Y$ such that 
${}^x\mathfrak{K} = G^{(x)}(\hat{\mathcal{O}}_K)$. Let $B^{(x)}/Y$ be a Borel 
subgroup of $G^{(x)}/Y$. Then, by \ref{piandnu}, \ref{GeometryOfNumbers}, 
and \ref{piandnu2}, $$\nu_i(B^{(x)},G^{(x)}(\hat{\mathcal{O}}_K)) = q^{n_i(B^{(x)})}.$$
Therefore, if $c_1 < -2g$, then by \ref{reduction1} the conclusion of the 
theorem holds for $C_1 := q^{c_1}$. 
\end{proof}

\begin{corollary}[{\cite[2.1.2]{Harder:1969}}] \label{reduction4}
Let $G/Y$ be a rationally trivial group $Y$-scheme, let 
$\mathfrak{K} := G(\hat{\mathcal{O}}_K)$, and let $C_1 = q^{c_1}$ be a 
constant for which the conclusion of \ref{reduction3} holds. Then there exist 
constants $C_2 > \Gamma > C_1$ such that the following hold: let 
$x \in G(\hat{\mathbb{A}}_K)$, let $B/Y$ be a Borel subgroup of $G/Y$ such 
that $\nu_{i}(B,{}^x\mathfrak{K}) \geq C_1$ for all $i \in I$, and 
let $\alpha_{i_0}$ be a simple root of $B$ such that 
$\nu_{i_0}(B,{}^x\mathfrak{K}) \geq C_2$. Then each Borel subgroup $B'/Y$ of 
$G/Y$ with $\nu_{i}(B',{}^x\mathfrak{K}) \geq C_1$ for all $i \in I$ 
satisfies $\nu_{i_0}(B',{}^x\mathfrak{K}) \geq \Gamma$ and is contained in
$P_{i_0}(B)$. 
\end{corollary}

\begin{proof}
By \ref{reduction2} the conclusion holds for $C_2 := q^{c_2}$ and $\Gamma := q^\gamma$. 
\end{proof}

\begin{corollary}[{\cite[2.3.2]{Harder:1969}}] \label{reduction5}
Let $G/Y$ be a reductive group $Y$-scheme and let 
$\mathfrak{K} := G(\hat{\mathcal{O}}_K)$. Then there exists a constant 
$C_1 > 0$ such that for each $x \in G(\hat{\mathbb{A}}_K)$ there exists a 
minimal $K$-parabolic subgroup $B/Y$ of $G/Y$ with $$\nu_i(B,{}^x\mathfrak{K}) 
\geq C_1 \quad \quad \mbox{ for all $i \in I$.}$$
\end{corollary}

\begin{definition} \label{reduceddef}
Once and for all fix $C_1 \in (0,1)$ such that the conclusion of \ref{reduction5} holds. A pair consisting of a minimal parabolic subgroup $B$ of $G/Y$ and an element $x \in G(\hat{\mathbb{A}}_K)$ is called {\bf reduced}, if $\nu_i(B,{}^x\mathfrak{K}) \geq C_1$ for all $i$.
\end{definition}

\begin{corollary}[{\cite[2.3.3]{Harder:1969}}] \label{reduction6}
Let $G/Y$ be a reductive group $Y$-scheme and let $\mathfrak{K} := 
G(\hat{\mathcal{O}}_K)$. Then there exist constants $C_2 > \Gamma > C_1$ 
such that the following hold: let $x \in G(\hat{\mathbb{A}}_K)$, let $B/Y$ be 
a minimal parabolic $K$-subgroup of $G/Y$ such that $(B,x)$ is reduced, and 
let $\alpha_{i_0}$ be a simple root of $B$ such that 
$\nu_{i_0}(B,{}^x\mathfrak{K}) \geq C_2$. Then each minimal parabolic 
$K$-subgroup $B'/Y$ of $G/Y$ with $(B',x)$ reduced satisfies 
$\nu_{i_0}(B',{}^x\mathfrak{K}) \geq \Gamma$ and $B' \subset P_{i_0}(B)$. 
\end{corollary}
Note that \ref{reduction5} and \ref{reduction6} follow from \ref{reduction3} and \ref{reduction4} by a standard field extension argument using \cite[2.3.5]{Harder:1969}.

\medskip
\newparagraph \label{finitecosetspace}
Let $G/Y$ be a reductive group $Y$-scheme, let $B/Y$ be a minimal 
parabolic $K$-subgroup of $G$, and let 
$\mathfrak{K} := G(\hat{\mathcal{O}}_K)$. Since 
$G(\hat{\mathbb{A}}_K)/B(\hat{\mathbb{A}}_K)$ is compact and 
$\mathfrak{K}$ is open (\cite[p.~36]{Harder:1969}), the double coset 
space $\mathfrak{K} \backslash G(\hat{\mathbb{A}}_K)/B(\hat{\mathbb{A}}_K)$ 
is finite. As $B(\hat{\mathbb{A}}_K)$ is self-normalizing in $G(\hat{\mathbb{A}}_K)$, one can consider $G(\hat{\mathbb{A}}_K)/B(\hat{\mathbb{A}}_K)$ as the space of conjugates of $B(\hat{\mathbb{A}}_K)$ in $G(\hat{\mathbb{A}}_K)$, and $\mathfrak{K} \backslash G(\hat{\mathbb{A}}_K)/B(\hat{\mathbb{A}}_K)$ as the space of $\mathfrak{K}$-orbits via conjugation on these. 

\begin{theorem}[{\cite[p.~40]{Harder:1969}}] \label{fundamentaldomain}
Let $G/Y$ be a reductive group $Y$-scheme, let $B/Y$ be a minimal parabolic 
$K$-subgroup of $G$, let $\mathfrak{K} := G(\hat{\mathcal{O}}_K)$, let 
$B^{(1)}$, ..., $B^{(t)}$ be a system of representatives of the 
$\mathfrak{K}$-orbit space $\mathfrak{K} \backslash 
G(\hat{\mathbb{A}}_K)/B(\hat{\mathbb{A}}_K)$ (cf.\ \ref{finitecosetspace}), 
let $\xi_s \in G(K)$ with $B^{(s)} = \xi_s^{-1} B^{(1)}\xi_s$, and for 
$c \in \mathbb{R}$ define $B^{(s)}(c) = \{ x \in B^{(s)}(\hat{\mathbb{A}}_K) 
\mid |\alpha_i(x)| \leq c \mbox{ for all $i \in I$} \}$.
Then there exists $r \in \mathbb{R}$ such that 
$$\bigcup_{s = 1}^t B^{(1)}(r)\xi_s\mathfrak{K}$$ is a fundamental domain 
for $G(K)\backslash G(\hat{\mathbb{A}}_K)$.
\end{theorem}

\begin{proof}
For $x \in G(\hat{\mathbb{A}}_K)$, by \ref{reduction5}, there exists a minimal parabolic $K$-subgroup $B$ of $G$ such that $\nu_i(B,{}^x\mathfrak{K}) \geq C_1$ for all $i \in I$. Let $B^{(s)}$ be the representative of the $\mathfrak{K}$-orbit of $x^{-1}Bx$ in $G(\hat{\mathbb{A}}_K)/B(\hat{\mathbb{A}}_K)$, let $u \in \mathfrak{K}$ such that $u^{-1}x^{-1}Bxu = B^{(s)}$, let $a \in G(K)$ with $aBa^{-1} = B^{(s)}$, and define $y:=axu \in N_{G(\hat{\mathbb{A}}_K)}(B^{(s)}(\mathbb{A}_K)) = B^{(s)}(\mathbb{A}_K)$. Since $a \in G(K)$
one has $$\nu_i(B,{}^x\mathfrak{K}) \stackrel{\ref{kinvariance}}{=} \nu_i({}^aB,{}^{ax}\mathfrak{K}) = \nu_i(B^{(s)},{}^{y}\mathfrak{K}).$$
By \ref{transformationformulafornu}, for each $i \in I$ one has $|\alpha_i(y)| = \nu_i(B^{(s)},\mathfrak{K}) \left(\nu_i(B^{(s)},{}^{y}\mathfrak{K})\right)^{-1} \leq \nu_i(B^{(s)},\mathfrak{K}) C_1^{-1}$. Hence, for $r := \max\{ \nu_i(B^{(s)},\mathfrak{K}) C_1^{-1} \mid 1 \leq s \leq t, i \in I \}$, to each $x \in G(\hat{\mathbb{A}}_K)$ there exists $a \in G(K)$ and $u \in \mathfrak{K}$ such that $ax = yu^{-1} \in B^{(s)}(r)\mathfrak{K} = \xi_s^{-1} B^{(1)}(r)\xi_s\mathfrak{K}$. The claim follows
because $a, \xi_s \in G(K)$.
\end{proof}

\section{Filtrations of Euclidean buildings}

\begin{introduction}
In this section I translate the geometric version of Harder's reduction
theory into the setting of Euclidean buildings based on \cite{Harder:1977}.
From this section on the survey is intended for the reader familiar with
the concept of Euclidean buildings, as simplicial complexes and as CAT(0) 
spaces. For both introductory and further reading the sources 
\cite{Abramenko/Brown:2008}, \cite{Bridson/Haefliger:1999}, 
\cite{Brown:1989}, \cite{Weiss:2003}, \cite{Weiss:2009} are highly recommended.
\end{introduction}

\newparagraph \label{building}
Let $Y/\mathbb{F}_q$ be a non-singular projective curve, let $S \subset Y^\circ$ be finite, and let $G/Y$ be a reductive group
$Y$-scheme. Clearly, each of \ref{reduction3}, \ref{reduction4}, \ref{reduction5}, \ref{reduction6} holds for $x \in G(\hat{\mathbb{A}}_S) \subset G(\hat{\mathbb{A}}_K)$ (cf.\ \ref{sadeles}). Since $$G(K \cap \hat{\mathbb{A}}_S)\backslash G(\hat{\mathbb{A}}_S)/G(\hat{\mathcal{O}}_K)
\stackrel{\ref{sadeles}}{\cong} G(\mathcal{O}_S) \backslash G(\prod_{P \in S} K_P)/G(\prod_{P \in S} \hat{\mathcal{O}}_{P,K}) = G(\mathcal{O}_S) \backslash \prod_{P \in S}
G(K_P)/G(\hat{\mathcal{O}}_{P,K}),$$
the functions $\pi_i(B,{}^x\mathfrak{K})$ and $\nu_i(B,{}^x\mathfrak{K})$ (cf.\ \ref{piandnu2}) allow one to define $G(\mathcal{O}_S)$-invariant (cf.\ \ref{kinvariance}) filtrations on the Euclidean building $X$ of $\prod_{P \in S} G(K_P)$. The group $G(\mathcal{O}_S)$ is called an {\bf $S$-arithmetic group}. The set of special vertices of $X$ is $X_v := \prod_{P \in S} G(K_P)/G(\hat{\mathcal{O}}_{P,K})$. The diagonal embedding of $K$ in $\prod_{P \in S} K_P$ (cf.\ \ref{sadeles}) yields a diagonal embedding of the spherical building of $G(K)$ into the spherical building at infinity of $X$ with respect to the complete system of apartments. Note that this embedding is in general not simplicial.

\begin{definition} \label{horoball}
Let $X$ be a CAT(0) space, e.g., a Euclidean building, and let $\gamma : [0,\infty) \to X$ be a unit speed geodesic ray, i.e., $d(\gamma(t),\gamma(0)) = t$ for all $t \geq 0$.
The function $$b_\gamma : X \to \mathbb{R}: x \mapsto \lim_{t \to \infty} (t-d(x,\gamma(t)))$$ is the {\bf Busemann function} with respect to $\gamma$.
Note that $t = d(\gamma(0),\gamma(t)) \leq d(\gamma(0),x) + d(x,\gamma(t))$ and
$t - d(x,\gamma(t))$ non-decreasing in $t$, so that the limit always exists. A linear reparametrization of a Busemann function is called a {\bf generalized Busemann function}. For a (generalized) Busemann function $b_\gamma$, a sub-level set of $-b_\gamma$ is a {\bf horoball centred at $\gamma(\infty)$}. The boundary of a horoball is a {\bf horosphere}, {centred} at $\gamma(\infty)$.
\end{definition}
Some sources define the Busemann function with respect to a geodesic ray as $b_\gamma : X \to \mathbb{R}: x \mapsto \lim_{t \to \infty} (d(x,\gamma(t))-t)$. This does not affect the concept of a horoball, i.e., in that case a horoball is defined as a sub-level set of $b_\gamma$.

\begin{theorem} \label{logarithm} \label{transformation} \label{formula1}
Let $P$ be a maximal $K$-parabolic and let $\mathfrak{K} := G(\mathcal{O}_K)$.
Then for each $g \in \prod_{P \in S} P(K_P)$ one has
$$\log_q(\pi(P,{}^g\mathfrak{K})) = \log_q(\pi(P,\mathfrak{K})) + \sum_{P \in
S} \mathrm{deg}(P)\nu_P(\chi_P(g)).$$
In particular, there exists a generalized Busemann function $p(P,\cdot) : X \to \mathbb{R}$ whose restriction to the set $X_v$ of special vertices of $X$ equals $\log_q(\pi(P,\cdot))$.
\end{theorem}

\begin{proof}
One has
\begin{eqnarray*}
\log_q(\pi(P,{}^g\mathfrak{K})) 
& \stackrel{\ref{transformationformula}}{=} & \log_q(\pi(P,\mathfrak{K})) + 
\log_q(|\chi_P(g)|^{-1}) \\ 
& \stackrel{\ref{sadeles}}{=} & \log_q(\pi(P,\mathfrak{K})) + 
\log_q(|\chi_P(g)|_S^{-1}) \\& \stackrel{\ref{defnorm}}{=} & \log_q(\pi(P,\mathfrak{K})) + 
\sum_{P \in S} \mathrm{deg}(P)\nu_P(\chi_P(g)).
\end{eqnarray*} 
The maximal $K$-parabolic $P$ corresponds to a vertex in the building of $G(K)$ and hence yields an element $\xi$ in the spherical building at infinity of $X$ (cf.\ \ref{building}). As $\chi_P$ is a multiple of the fundamental weight corresponding to $P$, any geodesic ray $\gamma$ in $X$ with $\gamma(\infty) = \xi$ provides a generalized Busemann function $b_\gamma = p(P,\cdot) : X \to \mathbb{R}$ as claimed. 
\end{proof}

\begin{definition} \label{pandn}
Let $B/Y$ be a minimal parabolic $K$-subgroup of $G/Y$, let $\{ \alpha_1, ..., \alpha_r \}$ be the simple roots of $B$, let $(P_i)_i$ be the maximal parabolic $K$-subgroups of $G$, of type $\alpha_i$, and, for each $i$, let $p(P_i,\cdot) : X \to \mathbb{R}$ be the generalized Busemann function from \ref{formula1}. 
For $x \in X$ define 
\begin{eqnarray*}
p_i(B,x) & := & p(P_i,x), \\
\overline{p}_i(B,x) & := & c_ip_i(B,x), \\
n_i(B,x) & := & \sum_{j=1}^r c_{i,j} p_j(B,x),
\end{eqnarray*}
where $(c_i)_{1 \leq i \leq r}$ is a family of positive real numbers such that each $c_i{p}_i(B,\cdot)$ is a Busemann function and $c_{i,j} \in \mathbb{Q}$
are such that $\alpha_i = \sum_{j=1}^r c_{i,j} \chi_j$ (cf.\ \ref{piandnu}).
\end{definition}

\begin{corollary} \label{logarithm2}\label{transformation2} \label{formula2}
Let $B/Y$ be a minimal parabolic $K$-subgroup of $G/Y$. For each $1 \leq i \leq r$ and each $g \in \prod_{P \in S} B(K_P)$
one has
$$n_i(B,{}^g\mathfrak{K}) = n_i(B,\mathfrak{K}) + \sum_{P \in S} \mathrm{deg}(P)\nu_P(\alpha_i(g)).$$
\end{corollary}

\begin{proof}
Combine \ref{transformationformulafornu} with \ref{pandn}. 
\end{proof}

\newparagraph\label{existenceofD}
The chambers of $X$ are compact and pairwise isometric and each contains 
a special vertex. Hence there exists $d > 0$ such that for each element $x \in X$ there exists a special vertex $x' \in X$ such that for each minimal parabolic $K$-subgroup $B$ of $G$ \begin{eqnarray*}
n_i(B,x) \geq c & \mbox{implies} & n_i(B,x') \geq c - d \quad \mbox{for all $i \in I$ and $c \in \mathbb{R}$ and,} \\
n_i(B,x') \geq c & \mbox{implies} & n_i(B,x) \geq c - d \quad \mbox{for all $i \in I$ and $c \in \mathbb{R}$.}
\end{eqnarray*}

\begin{theorem}[{\cite[1.4.2]{Harder:1977}}]\label{c1}
There exists $c_1 \in \mathbb{R}$ such that for each element $x \in X$ there exists a minimal parabolic $K$-subgroup $B$ of $G$ with $n_i(B,x) \geq c_1$ for all $i \in I$.
\end{theorem}

\begin{proof}
In case one only considers special vertices $x \in X$, such a constant $c_1:=\log_q(C_1)$ exists by \ref{reduction3}. If $d$ is a constant for which the conclusion of \ref{existenceofD} holds, then replacing $c_1$ by $c_1 - d$ therefore implies the assertion for arbitrary elements $x \in X$.
\end{proof}

\begin{theorem}[{\cite[1.4.4]{Harder:1977}}] \label{c2}
There exist constants $c_2 > \gamma > c_1$ such that for $x \in X$, a minimal $K$-parabolic $B$ with $n_i(B,x) \geq c_1$ for all $i \in I$, the family $(P_i)_{i \in I}$ of maximal parabolic $K$-subgroups of $G$ containing $B$, and $j \in I$ with $n_j(B,x) \geq c_2$, each minimal parabolic $K$-subgroup $B'$ of $G$ with $n_i(B',x) \geq c_1$ for all $i \in I$
 satisfies $n_j(B',x) \geq \gamma$, and is contained in $P_j$. 
\end{theorem}

\begin{proof}
In case one only considers special vertices $x \in X$, such constants $c_2:=\log_q(C_2)$ and $\gamma := \log_q(\Gamma)$ exist by \ref{reduction6}.
Choose again a 
constant $d$ for which the conclusion of \ref{existenceofD} 
holds, define $c_1' := c_1 - d$ and $C'_1 := q^{c_1'}$, use this constant $C'_1$ in \ref{reduceddef} to define reduced pairs, let $C'_2$ and $\Gamma'$ be constants for which the conclusion of \ref{reduction6} holds for this definition of a reduced pair, and let $c'_2 := \log_q(C'_2)$. For $c_2 := c'_2 + d$ by \ref{existenceofD} there exists 
a special vertex $x' \in X$ such that 
\begin{eqnarray*}
n_i(B,x) \geq c_1 \quad \mbox{implies} & n_i(B,x') \geq c_1' & \mbox{for all $i \in I$}, \\
n_i(B',x) \geq c_1 \quad \mbox{implies} & n_i(B',x') \geq c_1' & \mbox{for all $i \in I$, and} \\
n_j(B,x) \geq c_2  \quad \mbox{implies} & n_j(B,x') \geq c_2'.&
\end{eqnarray*}
We conclude from \ref{reduction6} that $B'$ is contained in $P_j$ and that $n_j(B',x') \geq \log_q(\Gamma')$, whence $n_j(B',x) \geq \log_q(\Gamma')-d =: \gamma$. 
\end{proof}

\begin{notation} \label{choiceofconstants}
Once and for all fix constants $c_1, c_2, \gamma \in \mathbb{R}$ such that $c_1$ is negative, $c_2 > \gamma > c_1$ and
such that the conclusions of \ref{c1} and \ref{c2} hold.
\end{notation}

\begin{definition} \label{reducedpair}\label{close}\label{isolatedparabolic}
A pair $(B,x)$ consisting of a minimal $K$-parabolic subgroup $B$ of $G$ and an element $x \in X$ such that $n_i(B,x) \geq c_1$ for all $i \in I$ is called {\bf reduced}.
For a minimal parabolic $K$-subgroup $B$ of $G$ and a maximal $K$-parabolic $P_j \supseteq B$, following \cite[p.~254]{Harder:1974}, an element $x \in X$ is called {\bf close to the boundary of $X$ with respect to $P_j$}, if $(B,x)$ is a reduced pair and $n_j(B,x) \geq c_2$.
An element $x \in X$ is called {\bf close to the boundary of $X$}, if there exists a maximal $K$-parabolic $P$ such that $x$ is close to the boundary of $X$ with respect to $P$.
For $x \in X$ close to the boundary of $X$, define 
$$P_x := \bigcap\{ P \subset G \mid \mbox{$x$ is close to the boundary of $X$ with respect to $P$}\}.$$ By \ref{c2} the group $P_x$ is a $K$-parabolic subgroup of $G$ (cf.\ \ref{isolatedparabolic2}). Following \cite[p.~138]{Harder:1968}, it is called the {\bf isolated parabolic subgroup of $G$ corresponding to $x$}. For each $K$-parabolic $Q \supset P_x$, the element $x \in X$ is called {\bf close to the boundary of $X$ with respect to $Q$}.
\end{definition}

\begin{proposition}[{\cite[p.~35]{Behr:2004}}] \label{inclusionatinfinity}
Let $x \in X$ be close to the boundary of $X$ and let $P_x$ be the corresponding isolated parabolic subgroup of $G$. Let $\gamma$ be a geodesic ray in $X$ with $\gamma(0) = x$ and whose end point lies in the simplex of the building at infinity corresponding to $P_x$.
Then each $y \in \gamma([0,\infty))$ is close to the boundary of $X$.
Moreover, one has $P_y = P_x$ and for each minimal $K$-parabolic $B$ the pair $(B,y)$ is reduced if and only if $(B,x)$ is reduced.
\end{proposition}

\begin{Proof}{(Bux, Gramlich, Witzel)}
First notice that by (\ref{orthogonal}) and (\ref{acute}) for each reduced pair $(B,x)$ and all $i \in I$ one has $n_i(B,y) \geq n_i(B,x)$. In particular, each reduced pair $(B,x)$ gives rise to a reduced pair $(B,y)$ and, moreover, $y$ lies close to the boundary of $X$ with respect to $P_x$. This implies $P_y \subseteq P_x$. 

Conversely, let $(B,y)$ be a reduced pair. Then by \ref{c2}
one has $B \subseteq P_y \subseteq P_x$. Let $(P_i)_{i \in I}$ be the family 
of maximal parabolic $K$-subgroups of $G$ containing $B$ and let $I' \subseteq I$ 
such that $P_x = \bigcap_{i \in I'} P_i$. If there exists $j \in I$ such 
that $n_j(B,y) \geq c_2$, but $P_j \not\supseteq P_x$, then 
$j \in I \backslash I'$. As for each $i \in I \backslash I'$ one has 
$n_i(B,x) = n_i(B,y)$, in particular $n_j(B,x) = n_j(B,y) \geq c_2$, and in view of \ref{isolatedparabolic} the pair $(B,x)$ cannot be reduced. As 
being a reduced pair is a closed condition (see \ref{reducedpair}),
 there therefore exists a minimal $a \in (0,\infty)$ such that $(B,\gamma(a))$
 is reduced. The first paragraph of this proof implies that $P_{\gamma(a)} \subseteq P_x$.
Thus \ref{c2} applied to the reduced pair 
$(B,\gamma(a))$ yields $n_i(B,\gamma(a)) \geq \gamma$, and hence
$n_i(B,\gamma(a)) > c_1$ by \ref{choiceofconstants}, for all 
$i \in I'$. As the $n_i(B,\gamma(b))$ are continuous in $b$ and as for each $i \in I \backslash I'$ one has 
$n_i(B,x) = n_i(B,\gamma(b)) = n_i(B,y) \geq c_1$, this contradicts 
the minimality of $a$. Therefore $(B,x)$ has to be reduced. 
Consequently, each $j \in I$ with the property that $n_j(B,y) \geq c_2$ satisfies 
$P_j \supseteq P_x$, whence $j \in I'$. Thus we have $P_y = P_x$. 
\end{Proof}

\begin{definition} \label{filtrationsc3}
For $c \in \mathbb{R}$ define 
\begin{eqnarray*}
X^n(c) & = & \{ x \in X \mid (B,x) \mbox{ reduced implies $n_i(B,x) \leq c$ for all $i \in I$} \}, \\
X^p(c) & = & \{ x \in X \mid (B,x) \mbox{ reduced implies $p_i(B,x) \leq c$ for all $i \in I$} \},
\quad \mbox{ and}\\
X^{\overline{p}}(c) & = & \{ x \in X \mid (B,x) \mbox{ reduced implies $\overline{p}_i(B,x) \leq c$ for all $i \in I$} \}.
\end{eqnarray*}
These filtrations are $G(\mathcal{O}_S)$-invariant (cf.\ \ref{kinvariance}) and $G(\mathcal{O}_S)$-cocompact (cf.\ \cite[2.2.2]{Harder:1969}). There exists $c_3 \in \mathbb{R}$ such that $X^n(c_2) \subseteq X^{\overline{p}}(c_3)$.
\end{definition}

\begin{proposition}[{\cite[Section 5]{Bux/Gramlich/Witzel:2011}}] \label{uniquedirectiontoinfinity}
Let $c \geq c_3$, let $x \in X \backslash X^{\overline{p}}(c)$ and let $\overline{x} \in X^{\overline{p}}(c)$
be an element at which the function $X^{\overline{p}}(c) \to \mathbb{R} : z \mapsto d(x,z)$ assumes a global minimum. Then $P_{\overline{x}} = P_x$. Furthermore, there exists a unique unit speed geodesic ray $\gamma_x^c : [0,\infty) \to X$ with $\gamma_x^c(0) = x$ along which the function $X \backslash X^{\overline{p}}(c) \to \mathbb{R} : x \mapsto d(x,X^{\overline{p}}(c))$ assumes its steepest ascent; its end point lies in the simplex at infinity corresponding to $P_x$.
\end{proposition}

The preceding proposition shows that one can measure the distance from the 
set $X^{\overline{p}}(c_3)$ in a neat way. It also shows that there is a 
substantial difference between $K$-rank $1$ and higher $K$-rank. Indeed, 
by \ref{uniquedirectiontoinfinity}, in the case of $K$-rank $1$ the boundary 
of $X^{\overline{p}}(c_3)$ consists of hypersurfaces of codimension one 
which must have pairwise empty intersections, as otherwise one would need 
non-minimal isolated $K$-parabolics, which do not exist. The following result 
makes this heuristic argument more concrete. 

\begin{theorem}[{\cite[3.7]{Bux/Wortman}}] \label{horo}
If $\mathrm{rk}_K(G) = 1$, then there exists a collection $\mathcal{H}$ of pairwise disjoint horoballs of $X$ such that $X^{\overline{p}}(c_3) = X \backslash \mathcal{H}$ is $G(\mathcal{O}_S)$-invariant and $G(\mathcal{O}_S)$-cocompact.
\end{theorem}

\begin{proof}
Since the functions $\overline{p}$ are Busemann functions (cf.\ 
\ref{formula1}, \ref{pandn}), there clearly exists a collection 
$\mathcal{H}$ of horoballs of $X$ such that $X^{\overline{p}}(c_3) = X 
\backslash \mathcal{H}$. By \ref{filtrationsc3} the set 
$X^{\overline{p}}(c_3)$ is $G(\mathcal{O}_S)$-invariant and 
$G(\mathcal{O}_S)$-cocompact. It therefore remains to prove that the 
horoballs in $\mathcal{H}$ can be chosen to be either disjoint or equal. 
Let $H_1, H_2 \in \mathcal{H}$ have non-trivial intersection and let $B_1$,
$B_2$ be the minimal $K$-parabolic subgroups corresponding to their
respective centres (cf.\ \ref{horoball}). Then there exists $x \in X$ such 
that $\overline{p}_\alpha(B_1,x), \overline{p}_\alpha(B_2,x) \geq c_3$, 
where $\alpha$ denotes the unique simple root. Therefore, by 
\ref{filtrationsc3}, $n_\alpha(B_1,x), n_\alpha(B_2,x) \geq c_2$, and by 
\ref{c2} one has $B_1 = B_2$. Hence, $H_1 \subseteq H_2$ or $H_2 \subseteq H_1$. 
If $H_1 \neq H_2$, one can remove one of the two from $\mathcal{H}$ without 
changing $X \backslash \mathcal{H}$.
\end{proof}

\section{Applications and conjectures}

\begin{introduction}
In this section I sketch the applicability of the geometric version of
Harder's reduction theory to the study of finiteness properties of 
$S$-arithmetic groups over global function fields and state a very general 
conjecture on isoperimetric properties of $S$-arithmetic groups over
arbitrary global fields. For reduction theory over number fields I strongly recommend \cite{Platonov/Rapinchuk:1994}.
\end{introduction}

\subsection{Finiteness properties of $S$-arithmetic groups}

\begin{definition}
A group $\Gamma$ is of {\bf type $F_m$}, if it admits a free action on a contractible CW complex $X$ with finitely many orbits on the $m$-skeleton of $X$.  
\end{definition}

A group action on a CW complex is called {\bf cellular}, if the action preserves the cell structure, and {\bf rigid}, if each group element that elementwise fixes the skeleton of a cell in fact elementwise fixes the whole cell.

\begin{theorem}[{\cite[1.1]{Brown:1987}}] \label{brown}
Let $m \in \mathbb{N}$, let $X$ be an $(m-1)$-connected CW complex, and let $\Gamma \to \mathrm{Aut}(X)$ act cellularly, rigidly and cocompactly on $X$ such that the stabilizer of each $i$-cell is of type $F_{m-i}$. Then $\Gamma$ is of type $F_m$.
\end{theorem}

\begin{theorem}[{\cite[7.7]{Bux/Wortman}}] \label{horoconnected}
Let $X = X_1 \times \cdots \times X_t$ be an affine building, decomposed into its irreducible factors, let $\partial X$ be the spherical building at infinity of $X$, let $\partial^j X$ be the spherical building at infinity of $\prod_{i=1, i \neq j}^t X_i$, and let $\xi \in \partial X \backslash \bigcup_{1 \leq j \leq t} \partial^j X$. Then any horosphere centred at $\xi$ (\ref{horoball}) is $(\mathrm{dim}(X)-2)$-connected.
\end{theorem}

\begin{theorem}[{\cite[8.1]{Bux/Wortman}}]
Let $K$ be a global function field, let $G$ be an absolutely almost simple $K$-group of $K$-rank $1$, let $\emptyset \neq S \subset Y^\circ$ be finite, let $X$ be the Euclidean building of $\prod_{P \in S} G(K_P)$, and let $m = \mathrm{dim}(X) = \sum_{P \in S} \mathrm{rk}_{K_P}(G)$. Then $G(\mathcal{O}_S)$ is of type $F_{m-1}$, but not $F_m$. 
\end{theorem}

\begin{proof}
The group $G(\mathcal{O}_S)$ clearly acts rigidly and cellularly on the (contractible) building $X$. By \ref{horo}, the group $G(\mathcal{O}_S)$ acts cocompactly on $X^{\overline{p}}(c_3) = X \backslash \mathcal{H}$ and, in view of \ref{building}, each of the horoballs $H \in \mathcal{H}$ is centred at some $\xi$ which satisfies the hypothesis of \ref{horoconnected}. Therefore, by \ref{horoconnected}, $X^{\overline{p}}(c_3)$ is $(n-2)$-connected. As cell stabilizers in $G(\mathcal{O}_S)$ are finite, whence of type $F_\infty$, the claim follows from \ref{brown}.
\end{proof}

For $K$-rank greater than $1$ this strategy cannot work, as \ref{horo} becomes false. It can be adapted, however, which leads to a couple of technical difficulties that have first been overcome in \cite{Bux/Gramlich/Witzel}, \cite{Witzel} for the $S$-arithmetic groups $G(\mathbb{F}_q[t])$ and $G(\mathbb{F}_q[t,t^{-1}])$ where $G$ is an absolutely almost simple $\mathbb{F}_q$-group of rank $n \geq 1$. In this situation one can in fact make use of the theory of Euclidean twin buildings, as in \cite{Abramenko:1996}.

\begin{theorem}[{\cite[A]{Bux/Gramlich/Witzel}, \cite[Main Theorem]{Witzel}}]
Let $G$ be an absolutely almost simple $\mathbb{F}_q$-group of rank $n \geq 1$. Then $G(\mathbb{F}_q[t])$ is of type $F_{n-1}$, but not $F_n$, and $G(\mathbb{F}_q[t,t^{-1}])$ is of type $F_{2n-1}$, but not $F_{2n}$.
\end{theorem}

Recently, a combination of the ideas developed in \cite{Bux/Wortman}, \cite{Bux/Gramlich/Witzel}, \cite{Witzel} with Harder's reduction theory allowed the authors of \cite{Bux/Gramlich/Witzel:2011} to prove the following theorem, providing a positive answer to the question asked in \cite[13.20]{Abramenko/Brown:2008}, \cite[p.~80]{Behr:1998}, \cite[p.~197]{Brown:1989}.

\begin{theorem}[{\cite[Rank Theorem]{Bux/Gramlich/Witzel:2011}}] 
\label{rank}
Let $K$ be a global function field, let $G$ be an absolutely almost simple $K$-isotropic $K$-group, let $\emptyset \neq S \subset Y^\circ$ be finite, let $X$ be the Euclidean building of $\prod_{P \in S} G(K_P)$, and let $m = \mathrm{dim}(X) = \sum_{P \in S} \mathrm{rk}_{K_P}(G)$. Then $G(\mathcal{O}_S)$ is of type $F_{m-1}$, but not $F_m$.
\end{theorem}

\begin{remark}
Using the notation introduced above, already \cite[1.1]{Bux/Wortman:2007} established that $G(\mathcal{O}_S)$ is not of type $F_{m}$.
\end{remark}

\begin{remark}
$S$-arithmetic groups over number fields are known to be of type $F_\infty$, cf.\ \cite[\S 11]{Borel/Serre:1976}. 
\end{remark}

\subsection{Isoperimetric properties of $S$-arithmetic groups}

In the number field case, results similar to \ref{reduction5} and \ref{reduction6} hold, cf.\ \cite{Borel/Harish-Chandra:1962}, \cite[p.~53]{Harder:1969}. 
To the best of my knowledge, the question posed by Harder (\cite[p.~54]{Harder:1969}) whether it is possible to prove the results from \cite{Borel/Harish-Chandra:1962} using methods similar (or at least closer) to the approach used by Harder is still open. The key, of course, would be to prove \ref{reduction2} for number fields.

As this problem by itself probably will not attract sufficient attention, I will finish this survey by stating a very general conjecture on properties of $S$-arithmetic groups over arbitrary global fields, i.e., global function fields or number
fields, which, if verified, provides another proof of \ref{rank}.

\begin{definition}
A {\bf coarse $n$-manifold} $\Sigma$ in a metric space $X$ is a function from the vertices of a triangulated $n$-manifold $M$ into $X$. The homeomorphism type of the manifold $M$ is called the {\bf topological type} of $\Sigma$. The {\bf boundary} $\partial \Sigma$ is the restriction to $\partial M$ of the function $\Sigma$. The coarse manifold $\Sigma$ has {\bf scale} $r \in \mathbb{R}_+$, if $d(\Sigma(x),\Sigma(y)) \leq r$ for all adjacent vertices $x$, $y$ of $M$. The {\bf volume} $\mathrm{vol}(\Sigma)$ equals the number of vertices in $M$.
\end{definition}

\begin{conjecture}[{\cite{Bux/Wortman:2007}}]    
Let $K$ be a global field, i.e., global function field or a number field, let $G$ be an absolutely almost simple $K$-isotropic $K$-group, let $S$ be a non-empty finite set of places of $K$ containing all archimedean ones, and let $n < \sum_{P \in S} \mathrm{rk}_{K_P}(G)$. 
To any $r_1 > 0$ there exists a linear polynomial $f$ and $r_2 > 0$ such that, if $\Sigma \subseteq \prod_{P \in S} G(K_P)$ is a coarse $n$-manifold of scale $r_1$ with $\partial\Sigma \subseteq G(\mathcal{O}_S)$, then there exists a coarse $n$-manifold $\Sigma' \subseteq G(\mathcal{O}_S)$ of scale $r_2$ and identical topological type as $\Sigma$ such that $\partial\Sigma' = \partial\Sigma$ and $\mathrm{vol}(\Sigma') \leq f(\mathrm{Vol}(\Sigma))$. 
\end{conjecture}

In case $n < |S|$, the existence of a polynomial $f$ of {\em unspecified degree} and $r_2 > 0$ as in the conjecture have been established by Bestvina, Eskin and Wortman. The precise statement of their result is as follows.

\begin{theorem}[{\cite[2]{Wortman}}]
Let $K$ be a global field, i.e., global function field or a number field, let $G$ be an absolutely almost simple $K$-isotropic $K$-group, let $S$ be a non-empty finite set of places of $K$ containing all archimedean ones, and let $n < |S|$. To any $r_1 > 0$ there exists a polynomial $f$ and $r_2 > 0$ such that, if $\Sigma \subseteq \prod_{P \in S} G(K_P)$ is a coarse $n$-manifold of scale $r_1$ with $\partial\Sigma \subseteq G(\mathcal{O}_S)$, then there exists a coarse $n$-manifold $\Sigma' \subseteq G(\mathcal{O}_S)$ of scale $r_2$ and identical topological type as $\Sigma$ such that $\partial\Sigma' = \partial\Sigma$ and $\mathrm{vol}(\Sigma') \leq f(\mathrm{Vol}(\Sigma))$.
\end{theorem}

Note that this result provides an alternative proof that the $S$-arithmetic 
group 
$G(\mathcal{O}_S)$
in \ref{rank} is of type $F_{|S|-1}$.

\begin{footnotesize}

\bibliographystyle{ralf}
\bibliography{habil}

\end{footnotesize}

\vspace{2cm}

\noindent
Ralf Gramlich \\
Fachbereich Mathematik \\
TU Darmstadt \\
Schlo\ss gartenstra\ss e 7 \\
64289 Darmstadt \\
Germany \\
e-mail: {\tt gramlich@mathematik.tu-darmstadt.de} 

\medskip \noindent
Justus-Liebig-Universit\"at Gie\ss en \\
Mathematisches Institut \\
Arndtstra\ss e 2 \\
35392 Gie\ss en \\
Germany

\end{document}